\documentclass[9pt]{extarticle} 
\usepackage[ 
  paperwidth=5.88in,paperheight=8.60in,
  textwidth=4.34in,textheight=6.7in]{geometry} 

\usepackage{fancyhdr} 
\usepackage{amsmath, amsthm, amsfonts, amssymb, mathdots}
\usepackage{stmaryrd}
\usepackage{paralist}
\usepackage{framed}
\usepackage{mathtools}
\usepackage{tikz}
\usepackage{graphicx}
\usepackage{hyperref}
\usepackage{cleveref}
\usepackage[english,ruled,lined,linesnumbered]{algorithm2e}
\usepackage{mathtools}
\newtheorem{theorem}{Theorem}

\newtheorem{lemma}{Lemma}

\theoremstyle{definition}
\newtheorem{remark}{Remark}

\fancypagestyle{headings}{%
\fancyhf{} 
\fancyfoot[L]{%
}

}

\makeatletter
\renewcommand{\maketitle}{
\pagestyle{plain}
\vspace*{\baselineskip}
\begin{center}

\MakeUppercase{\small Tom Werner} \\ \emph{Institute for Numerical Analysis, TU Braunschweig} \\ \emph{Universit\"atsplatz 2, 38106 Braunschweig, Germany} \\ (E-Mail: tom.werner@tu-braunschweig.de) 

\vspace*{3\baselineskip}

\MakeUppercase{\large \@title}
\end{center} \vskip-\baselineskip
}
\makeatother

\usepackage{titlesec} 
\titleformat{\section}[hang]{\Large\scshape}{\thesection. }{0pt}{\centering}[]
\titleformat{\subsection}[hang]{\large\scshape}{\thesubsection. }{0pt}{\centering}[]
\titleformat{\subsubsection}[hang]{\scshape}{\thesubsubsection. }{0pt}{\centering}[]

\title{On using the complex step method for the approximation of Fréchet derivatives of matrix functions in automorphism groups}
\begin{document}
\maketitle

\begin{abstract}
\noindent
We show, that the Complex Step approximation to the Fréchet derivative of real matrix functions is applicable to the matrix sign, square root and polar mapping using iterative schemes. While this property was already discovered for the matrix sign using Newton's method, we extend the research to the family of Padé iterations, that allows us to introduce iterative schemes for finding function and derivative values while approximately preserving automorphism group structure. 
\end{abstract}
\section{Introduction}
Matrix functions $f:\mathbb{C}^{m,n}\rightarrow\mathbb{C}^{m,n}$ are of increasing interest for many practical applications arising in mechanics, electrodynamics, statistics, pyhsics and mathematics \cite[Sec. 2]{matrixfunctions}. While a few matrix problems have scalar or vector-valued equivalents with well-studied solution techniques, that can be used for matrices with minor adjustments as well, there are also a lot of issues that are unique to matrix problems. These include eigenvalue or singular value problems, the non-commutativity of matrix multiplication or additional properties such as orthogonality or sparsity of desired solutions.\\
In this paper, we will focus on the matrix square root, sign and polar function that are widely used for solving subproblems arising from the applications mentioned before, including nonlinear matrix equations, Ricatti- and Sylvester equations and eigenvalue problems \cite[Sec. 2]{matrixfunctions}. For these functions, we present a variety of direct and iterative methods for evaluating them at a given point.\\
As in some of the applications mentioned above the problem formulation provides a special type of automorphism structure, we address the question whether or not our direct and iterative methods are able to preserve that given structure. In our case, it turns out that the iterative methods have to be chosen carefully since some methods are able to preserve the structure while others are not.\\
The third topic we discuss is how to compute the Fréchet derivative $L_f(A,E)$ of a matrixfunction $f$ at a given matrix $A$ in direction of a given matrix $E$. The (Fréchet) derivative of a function is of interest in pratical applications as it enables a way to compute or estimate the functions condition number, which is a key parameter for sensitivity analysis. In case of data-based applications, the condition number of $f$ indicates, how sensitive the mathematical model is to inaccuracy of the input data and therefore expose, how accurat the input data has to be measured to obtain useful results from the solution. To compute the Fréchet derivatives of our three main functions, we present a direct approach and two possibilities for modifying the existing iterative methods in a way that allows us to compute $f(A)$ and $L_f(A,E)$ simultaneously.\\
The main task in this paper is to present a way for computing $f(A)$ and $L_f(A,E)$ using a single complex iteration while also preserving automorphism structure throughout the iteration. The presented class of iterative schemes uses the so-called Complex Step approximation \begin{equation*}
L_f(A,E)\approx \mathrm{Im}\left(\frac{f(A+ihE)}{h}\right)
\end{equation*}
to $L_f(A,E)$, where $h$ is a real scalar, $A$ and $E$ are real matrices and $f$ is a matrix function, in combination with our iterative methods introduced before.\\
The following chapters are organized as follows: First, we recall some results and definitions we need for our later discussions. Those include the definition of the sign, square root and polar function, some properties of the Fréchet derivative of matrix functions and the matrix valued Newton iteration.\\
In the third chapter, we present some earlier results concerning iterative methods for computing the three functions. We will first have a look at standard Newton methods and then move over to the family of Padé iterations, which will provide us a tool for preserving automorphism structure in our iteration. To conclude the section, we compare a direct and an iterative approach for computing the Fréchet derivative of any of the functions.\\
In the fourth chapter, we introduce the Complex Step (CS) approach for scalar and matrix functions as a way for approximating the Fréchet derivative. We apply the CS to the Newton methods discussed in the third chapter and show that this allows us to iteratively compute the Fréchet derivative of any of the functions.\\
Our main result is about the combination of the CS and the family of Padé iterations and will be stated in \Cref{thm:CSpade} in the fifth chapter. We will show that combining the CS and the Padé schemes introduced before, we are able to compute $f(A)$ and $L_f(A,E)$ simultaneously in a single iteration by extracting the real and imaginary parts of the iterates while also preserving automorphism structure in the real part of the iterates.\\
To conclude our work, we present a few numerical tests emphasizing our theoretical results from before. The results are then analyzed and discussed and possible starting points for further research are mentioned.      
\section{Preliminaries, notation and definitions}
\subsection{Preliminaries and notation}
Throughout this paper, we will use the following notation. We denote the set of real or complex $m\times n$ matrices by $\mathbb{R}^{m,n}$ or $\mathbb{C}^{m,n}$, respectively, and by $\mathbb{R}^-$ the set of real non-positive scalars $(-\infty,0]$. In the case of complex values $x$, we use $\mathrm{Re}(x)$ and $\mathrm{Im}(x)$ for the real and imaginary part, respectively and apply the notation equally to matrices. We denote matrices in capital letters, where for a given matrix $X=\left(x_{ij}\right)_{i,j}$, the expressions $X^T$, $X^H=\overline{X}^T$, $X^{-1}$ and $X^+$ stand for the transpose, conjugate transpose, inverse and Moore-Penrose inverse of $X$, respectively. The identity matrix is denoted $I$ or $I_n$, if the size is not clear from the context. Additionaly, we use $\mathrm{O}_n(\mathbb{R})$ or $\mathrm{O}_n(\mathbb{C})$ for the set of orthogonal $n\times n$ matrices (i.e. $X^TX=I_n$) over $\mathbb{R}$ or $\mathbb{C}$ and $\mathrm{U}_n(\mathbb{C})$ for the set of unitary matrices (i.e. $X^HX=I_n$). The spectrum $\sigma(X)$ stands for the set of eigenvalues of $X$ and $||X||_F=\sqrt{\left(\sum_{i,j}x_{ij}^2\right)}$ is the Frobenius-Norm, which we will mostly use for convergence analysis. To characterize convergence and approximation properties, we use the Landau- or O-Notation in the following way: we say that $f=\mathcal{O}(g)$ or $f\in \mathcal{O}(g)$, if 
\begin{equation*}
\limsup_{x\rightarrow a}\left|\frac{f(x)}{g(x)}\right| <\infty.
\end{equation*} 
That is, we have $f=\mathcal{O}(g)$ if there exists a constant $C>0$ such that $|f(x)|\leq C |g(x)|$ for $x\rightarrow a$. 
We will mostly need the case $x\rightarrow 0$ for our application. For matrix functions $f$ and $g$, we use the Frobenius-Norm for $f$ and $g$. Furthermore, we say that $f=o(g)$ or $f\in o(g)$, if 
\begin{equation*}
\limsup_{x\rightarrow a}\left|\frac{f(x)}{g(x)}\right| = 0.
\end{equation*}   
Concerning matrix decompositions, we will need the singular value decomposition $A=U\Sigma V^H$ for complex rectangular matrices and the Jordan canonical form $A=ZJZ^{-1}$ for complex square matrices.
\subsection{The square root, sign and polar function}\label{sec:functions}
\subsubsection{The Square Root Function} 
Given a square matrix $A\in\mathbb{C}^{n,n}$, any matrix $X\in\mathbb{C}^{n,n}$ that satisfies $X^2 = A$ is a square root of $A$. While, in the scalar case, this definition secures uniqueness of the square root (up to its sign), there can exist infinitely many square roots of $A$ in the matrix case, if no further restrictions are impsoed on $A$ or $X$. A common example is the family of nilpotent matrices of degree two, that are all roots of the matrix $0_{n,n}$ of all zeros.\\
However, if $A$ has no eigenvalues on $\mathbb{R}^-$, it can be shown that there is only one solution $X$ to $A=X^2$ with eigenvalues entirely in the open right half-plane (that is, $\mathrm{Re}(\lambda)>0$, for every $\lambda\in\sigma(X)$). This solution is called the principal square root of $A$ and is denoted $X=A^{1/2}$. The principal square root satisfies the following properties\cite[Thm. 1.13, 1.18 + 1.29]{matrixfunctions}:
\begin{enumerate}
\item $A\in\mathbb{R}^{n,n}$ $\Longrightarrow$ $A^{1/2}\in\mathbb{R}^{n,n}$
\item $\left(A^T\right)^{1/2}=\left(A^{1/2}\right)^T$ and $\left(A^H\right)^{1/2}=\left(A^{1/2}\right)^H$
\item $\left(A^{-1}\right)^{1/2}=\left(A^{1/2}\right)^{-1}\equiv A^{-1/2}$
\end{enumerate}
Additionaly, the principal square root function is analytic, hence continuous, on the set of matrices with no eigenvalues on $\mathbb{R}^-$\cite{analytic1}.\\
The square root is frequently used for solving definite generalized eigenvalue problems and for computing solutions to quadratic matrix equations, such as algebraic Riccati equations \cite[Sec. 2]{matrixfunctions,analytic1,analytic2}
\subsubsection{The Sign Function}
In the scalar case, the sign function maps a complex value to the sign of its real part. That is, for $x\in\mathbb{C}$ with $\mathrm{Re}(x)\neq 0$, the sign of $x$ is given by
\begin{equation}
\mathrm{sign}(x) = \begin{cases}1, & \mathrm{Re}(x)>0.\\
								-1 & \mathrm{Re}(x)<0.\end{cases}
\end{equation}
In the matrix case, given a Jordan canonical form $A=ZJZ^{-1}$ of $A$ with $J=\mathrm{diag}(J_1,J_2)$, the sign of $A$ is defined as
\begin{equation}
\mathrm{sign}(A)=\mathrm{sign}\left(Z\begin{bmatrix}J_1 & 0\\0 & J_2\end{bmatrix} Z^{-1}\right) = Z \begin{bmatrix}-I_p & 0\\ 0 & I_{n-p}\end{bmatrix} Z^{-1},
\end{equation}
where $J_1\in\mathbb{C}^{p,p}$ contains the Jordan blocks for eigenvalues of $A$ with negative real parts and $J_2\in\mathbb{C}^{(n-p),(n-p)}$ contains the blocks for eigenvalues with positive real parts. In case of a diagonalizable $A$, $\mathrm{sign}(A)$ can be seen as computing the scalar sign of every eigenvalue of $A$ seperately. Notice that, in analogy to the scalar definition, purely imaginary eigenvalues of $A$ are not feasible for this definition.\\
The matrix sign function is used in control theory as a tool for solving a particular type of Lyapunov and Riccati equations arising from discretization or for counting eigenvalues in designated areas of the complex plain \cite[Sec. 2]{matrixfunctions}.     
\subsubsection{The Polar Function}
Given a rectangular matrix $A\in\mathbb{C}^{m,n}$ with $m\geq n$, $A$ can be decomposed as 
\begin{equation}
	A = QP,\quad Q\in\mathbb{C}^{m,n},\quad P\in\mathbb{C}^{n,n}, \label{eq:polar}
\end{equation}
where $P$ is Hermitian positive semidefinite and $Q$ has orthonormal columns. A decomposition as in \eqref{eq:polar} is called polar decomposition of $A$, with $Q$ being called the unitary polar factor and $P$ the Hermitian polar factor. It can be shown that $P$ is unique in any polar decomposition of $A$ and explicitly given by $P=\bigl(A^HA\bigr)^{1/2}$. Moreover, if $A$ has full rank ($\mathrm{rank}(A)=n$), $P$ is positive definite and $Q$ is unique as well. A polar decomposition can be computed taking the left and right singular vectors of an economy sized SVD $A=U\Sigma V^H$ of $A$ and defining $Q=UV^H$ and $P=V\Sigma V^H$ \cite[Sec. 8]{matrixfunctions}. In the following, we use the notation $Q=\mathcal{P}(A)$ for the unitary polar factor. The polar decomposition is of special interest in approximation applications, as the unitary and Hermitian factors lead to the closest unitary and Hermitian semidefinite approximations to $A$ \cite[Sec. 8.1]{matrixfunctions}.
\subsubsection{Connections between the square root, sign and polar function}\label{sec:connection}
The three functions introduced in \Cref{sec:functions} are connected to each other by the following formulae: Given a method for evaluating the principal square root for a given matrix, one can compute the sign and polar factor of that matrix by the identities
\begin{equation}
	\mathrm{sign}(A)=A\left(A^2\right)^{-1/2}\quad\text{and}\quad \mathcal{P}(A)=A\left(A^HA\right)^{-1/2} \label{eq:sqrtidentities}
\end{equation} 
using only elementary matrix operations and the principal square root (see \cite[Sec. 2]{highamsign}).
\begin{remark} \label{rem:signpolarsqrt}
Note that $A$ having no eigenvalues on the imaginary axis, which is required for the existence of $\mathrm{sign}(A)$, is equivalent to $A^2$ having no eigenvalues on $\mathbb{R}^-$, which is required for $A^{1/2}$ to exist. Similarly, if $A\in\mathbb{C}^{m,n}$ has full rank, $A^HA$ is Hermitian positive definite and, as such, does not have eigenvalues on $\mathbb{R}^-$, making \eqref{eq:sqrtidentities} well defined.\\
Additionally, we can see that the sign and polar mapping are continuous on the set of matrices with no imaginary eigenvalues and matrices of full rank, respectively, since they are compositions of continuous functions on these sets \cite{analytic2}.
\end{remark}
\noindent
On the other hand, if one has access to an algorithm for computing the sign of a given matrix, Higham\cite[Sec. 5]{matrixfunctions} has shown that, given $A,B\in\mathbb{C}^{n,n}$ with $AB$ having no eigenvalues on $\mathbb{R}^-$, 
\begin{equation}
\mathrm{sign}\left(\begin{bmatrix}0 & A\\ B & 0\end{bmatrix}\right)=\begin{bmatrix}0 & C\\ C^{-1} & 0\end{bmatrix} \label{eq:blocksign}
\end{equation} 
holds, where $C=A(BA)^{-1/2}$. Using \eqref{eq:blocksign} and choosing $B=I$ leads to the expression
\begin{equation}
\mathrm{sign}\left(\begin{bmatrix}0 & A\\ I & 0\end{bmatrix}\right)=\begin{bmatrix}0 & A^{1/2}\\ A^{-1/2} & 0\end{bmatrix},
\end{equation}
for the principal square root and inverse square root, whereas choosing $B=A^H$ yields the formula
\begin{equation}
\mathrm{sign}\left(\begin{bmatrix}0 & A\\ A^H & 0\end{bmatrix}\right)=\begin{bmatrix}0 & \mathcal{P}(A)\\ \mathcal{P}(A)^H & 0\end{bmatrix}
\end{equation}
for the unitary polar factor. These relations are important and should be kept in mind when designing efficient algorithms for evaluating any of the three functions.
\subsection{Fréchet derivatives}
As the set of real or complex matrices is a Banach space, differentiation of matrix functions is usually done in the sense of Fréchet derivatives. That is, given a matrix function $f:\mathbb{C}^{m,n}\rightarrow\mathbb{C}^{m,n}$, its Fréchet derivative $L_f(X)$ at a given point $X\in\mathbb{C}^{m,n}$ is a linear mapping $L_f(X):\mathbb{C}^{m,n}\rightarrow\mathbb{C}^{m,n}$, $E\mapsto L_f(X)(E)=L_f(X,E)$ that satisfies 
\begin{equation}
	f(X+E)-f(X)-L_f(X,E)=o(||E||), \label{eq:frechet}
\end{equation}
for every $E\in\mathbb{C}^{m,n}$. The Fréchet derivative, if it exists, can be shown to be unique\cite[Sec. 3.1]{matrixfunctions}.\\ 
More generally, the $k$-th Fréchet derivative of $f$ at a given point $X$ is the multilinear mapping 
\begin{align*}
&L_f^{[k]}(X):\mathbb{C}^{m,n}\times\cdots\times\mathbb{C}^{m,n}\rightarrow\mathbb{C}^{m,n},\\
&(E_1,\dots,E_k)\mapsto L_f^{[k]}(X;E_1,\dots,E_k),
\end{align*} 
that satisfies the recurrency
\begin{multline}
	L_f^{[k-1]}(X+E_k;E_1,\dots,E_{k-1})-L_f^{[k-1]}(X;E_1,\dots,E_{k-1})\\ -L_f^{[k]}(X;E_1,\dots,E_k)=o(||E_k||), \label{eq:higherfrechet}
\end{multline}
for every $E_1,\dots,E_k\in\mathbb{C}^{m,n}$, where $L_f^{[1]}(X)\equiv L_f(X)$ denotes the first Fréchet derivative obtained from \eqref{eq:frechet}\cite[Sec. 2]{highamfrechet2}.\\
For composite matrix functions, the following rules of differentiation apply:
\begin{enumerate}
\item For $g$ and $h$ being Fréchet differentiable at $A\in\mathbb{C}^{m,n}$ and $f=\alpha g + \beta h$, we have 
\begin{equation*}
	L_f(A,E)=\alpha L_g(A,E) + \beta L_h(A,E).  \tag{sum rule}
\end{equation*}
\item For $g$ and $h$ being Fréchet differentiable at $A\in\mathbb{C}^{m,n}$ and $f=g \cdot h$, we have 
\begin{equation*}
	L_f(A,E)=g(A)L_h(A,E) + L_g(A,E) h(A). \tag{product rule}
\end{equation*}
\item For $h$ being Fréchet differentiable at $A\in\mathbb{C}^{m,n}$, $g$ being Fréchet differentiable at $h(A)$ and $f=g\circ h$, we have 
\begin{equation*}
	L_f(A,E)=L_g(h(A),L_h(A,E)). \tag{chain rule}
\end{equation*}
\item If $f$ and $f^{-1}$ both exist and are continuous in a neighbourhood of $A\in\mathbb{C}^{n,n}$ and $f(A)$, respectively, and $L_f$ exists and is nonsingular at $A$, we have 
\begin{equation*}
	L_{f^{-1}}(f(A),E)=L_f^{-1}(A,E). \tag{inverse rule}
\end{equation*}
\end{enumerate} 
The Fréchet derivative is a strong tool when working with matrix functions as it allows to perform sensitivity analysis and enables a wide class of iterative methods for solving nonlinear matrix equations, such as Newton's or Halley's method \cite{halley2,newton,halley}.
\subsection{Newton Iterations}\label{sec:basenewton}
The most common approach for evaluating matrix functions iteratively is to use variants of Newton's method applied to a properly chosen target function. Taking the matrix square root as an example, it is immediate to see that any root $X^*\in\mathbb{C}^{n,n}$ of the quadratic function $f(X)=X^2-A$ is a square root of the given matrix $A\in\mathbb{C}^{n,n}$.\\
Recall that, for a vector valued function $g:\mathbb{C}^n \rightarrow \mathbb{C}^n$, the iterates $\{x_k\}_{k=0,1,...}$ in Newton's method \cite[Sec. 2.11]{numericalanalysis} are computed using the update rule
\begin{equation}
	x_{k+1}=x_k + d_k,\quad\text{where }d_k\text{ solves}\quad J_g(x_k) d = -f(x_k), \label{eq:vecnewton}
\end{equation}  
assuming nonsingularity of the Jacobian $J_g(x_k)$ for every iterate $x_k$. Under some mild assumptions, the iterative scheme \eqref{eq:vecnewton} converges quadratically to a root $x^*\in\mathbb{C}^n$ of $g(x)$, given an initial guess $x_0\in\mathbb{C}^n$ that is sufficiently close to $x^*$ \cite[Sec. 2.11]{numericalanalysis}.\\
For Fréchet differentiable matrix functions $G:\mathbb{C}^{m,n}\rightarrow\mathbb{C}^{m,n}$, the update from \eqref{eq:vecnewton} can be directly transferred to the matrix setting by replacing the linear system of equations $J_g(x_k)d_k=-g(x_k)$ with the matrix equation $L_G(X_k,D_k)=-G(X_k)$. The resulting Newton update
\begin{equation}
	X_{k+1}=X_k + D_k, \quad X_0\in\mathbb{C}^{m,n},  \label{eq:matnewton}
\end{equation}
where $D_k$ solves the matrix correction equation
\begin{equation}
	L_G(X_k,D) = -G(X_k), \label{eq:correction}
\end{equation}	
is again locally quadratically convergent under mild assumptions \cite{kantorovich}.\\
To be able to use Newton's method to its full effect, one has to be able to solve \eqref{eq:correction} efficiently in every Newton step. This task can be challenging if $L_G(X_k)$ is not known explicitly or hard to evaluate. However, we will see in the following section, that the update $D_k$ can be given in closed form for the three functions from \Cref{sec:functions}.
\section{Iterative approaches for evaluating the square root, sign and polar function}
\subsection{Quadratic Newton Iterations}\label{sec:quadnewton}
As we have seen before, defining $f(X)=X^2-A$, any root of $f$ is a square root of the given matrix $A\in\mathbb{C}^{n,n}$. In particular, assuming that $A$ has no eigenvalues on $\mathbb{R}^-$ and choosing the initial value for Newton's method to be $X_0=A$, the Newton square root sequence 
\begin{equation}
X_{k+1}=\frac{1}{2}\left(X_k+X_k^{-1}A\right),\quad X_0=A, \label{eq:newtonsqrt}
\end{equation} 
can be shown to converge to the principal square root $A^{1/2}$ of $A$ quadratically \cite[Thm. 6.9]{matrixfunctions}.\\
For the matrix sign, we can use the fact that $S=\mathrm{sign}(A)$ is involutary (i.e. $S^{-1}=S$) for a matrix $A\in\mathbb{C}^{n,n}$ with no imaginary eigenvalues \cite[Thm. 5.1]{matrixfunctions}. Thus, defining $g(X)=X^2-I_n$, the matrix $S$ satisfies $g(S)=0$. Applying Newton's method to $g$ and initializing the sequence with $X_0=A$, one obtains the Newton sign iteration 
\begin{equation}
X_{k+1}=\frac{1}{2}\left(X_k+X_k^{-1}\right),\quad X_0=A, \label{eq:newtonsign}
\end{equation} 
which is again known to be quadratically convergent to $S=\mathrm{sign}(A)$\cite[Thm. 5.6]{matrixfunctions}. In addition, by \eqref{eq:blocksign}, we have that \eqref{eq:newtonsign} converges to the limit 
\begin{equation}
S=\begin{bmatrix}0&A^{1/2}\\ A^{-1/2} & 0\end{bmatrix}\text{ if we choose }  
X_0=\begin{bmatrix}0&A\\ I_n & 0\end{bmatrix}. \label{eq:DBIinit}
\end{equation} 	
The resulting Newton update can be performed exclusively on the $(1,2)$- and $(2,1)$-block of the $2n$-by-$2n$-matrices $X_k$, which leads to the well-known Denman-Beavers-iteration (DB)
\begin{equation}
\begin{array}{lll}
Y_{k+1}&=\frac{1}{2}\left(Y_k+Z_k^{-1}\right), &Y_0=A,\\
Z_{k+1}&=\frac{1}{2}\left(Z_k+Y_k^{-1}\right), &Z_0=I_n,
\end{array}
\label{eq:DBI}
\end{equation}
that yields iterates $Y_k$ and $Z_k$ that converge to $A^{1/2}$ and $A^{-1/2}$, respectively \cite[Sec. 6.3]{matrixfunctions,dbi}. The quadratic convergence speed of \eqref{eq:DBI} is immediate from the derivation via \eqref{eq:newtonsign}.\\
For the polar function, let us first consider real matrices $A$. In this case, we may use the fact that $Q=\mathcal{P}(A)$ is, for any matrix $A\in\mathbb{R}^{m,n}$ having full rank $n$, the unique matrix with orthogonal columns that is closest to $A$ \cite[Thm. 8.4]{matrixfunctions}. As a consequnce, $Q$ is the root of $h(X)=X^TX-I_n$ that is closest to $A$. For square nonsingular matrices $A$, the typical Newton procedure as introduced in \Cref{sec:basenewton} leads to the iteration 
\begin{equation}
X_{k+1}=\frac{1}{2}\left(X_k+X_k^{-T}\right),\quad X_0=A, \label{eq:newtonpolarsquarereal}
\end{equation} 
that can be shown to converge quadratically to the orthogonal factor $Q$ \cite[Thm. 8.12]{matrixfunctions}. For rectangular full rank $A$, one has to replace the inverse $X_k^{-1}$ by the pseudoinverse $X_k^+$ to obtain the iteration
\begin{equation}
X_{k+1}=\frac{1}{2}\left(X_k+(X_k^+)^{T}\right)=\frac{1}{2}X_k\left(I_n+(X_k^TX_k)^{-1}\right),\text{ } X_0=A, \label{eq:newtonpolarreal}
\end{equation} 
which shares the convergence behaviour of \eqref{eq:newtonpolarsquare}\cite{rectpolar}.\\
For the case of complex matrices $A$, we want to emphasize that the mapping $X\mapsto X^TX$ is Fréchet differentiable (over $\mathbb{C}^{m,n}$) with 
\begin{equation*}
	L(A,E)=A^TE + E^T A,
\end{equation*} 
while the mapping $X\mapsto X^HX$ is not! This is due to the fact that linearity in  
\begin{equation*}
	L(A,\alpha E)=A^H\alpha E + (\alpha E)^H A=\alpha A^H E + \bar{\alpha} E^H A\neq \alpha L(A,E)
\end{equation*} 
is lost when conjugating the scaled direction $\alpha E$. However, it has been shown that defining $h(X)=X^HX-I_n$ for complex matrices and performing updates that are analogous to \eqref{eq:newtonpolarsquarereal} and \eqref{eq:newtonpolarreal}, one can use the iterative schemes
\begin{equation}
X_{k+1}=\frac{1}{2}\left(X_k+X_k^{-H}\right),\quad X_0=A, \label{eq:newtonpolarsquare}
\end{equation}  
and
\begin{equation}
X_{k+1}=\frac{1}{2}\left(X_k+(X_k^+)^{H}\right)=\frac{1}{2}X_k\left(I_n+(X_k^HX_k)^{-1}\right),\text{ } X_0=A, \label{eq:newtonpolar}
\end{equation} 
for computing the unitary factor $Q=\mathcal{P}(A)$ of the polar decomposition, even if $A$ is complex \cite[Thm. 8.4]{matrixfunctions}.  
\subsection{Iterations based on Padé approximations}
\subsubsection{Automorphism groups}
In many applications, the matrices of interest share a certain underlying structure, whether it is nonnegativity, sparsity or symmetry. In this work, we are particularly interested in automorphism group structure associated with a given bi- or sesquilinear form. That is, we aim to have a closer look at matrices $A\in\mathbb{C}^{n,n}$ that satisfy 
\begin{equation}
	A^H M A = M \quad \text{and/or} \quad A^T M A = M, \label{eq:automorphism}
\end{equation} 
for a given symmetric and nonsingular matrix $M$, which is equivalent to $A$ being an automorphism concerning the bi- or sesquilinear form 
$\langle\cdot,\cdot\rangle_M$ induced by $M$.  (i.e. $\langle Ax, Ay\rangle_M = \langle x,y\rangle_M$, for all $x,y\in\mathbb{C}^n$).\\
Note that the simplest case $M=I$ yields the standard euclidean scalar product, which leads to the group of orthogonal/unitary matrices. In our case, we are interested in the choices 
\begin{equation*}
J_{2n}=\begin{bmatrix} & I_n \\ -I_n &\end{bmatrix},\text{ }\Sigma_{p,q}=\begin{bmatrix} I_p & \\ & -I_q \end{bmatrix},\text{ }R_n=\begin{bmatrix} & & 1 \\ & \iddots & \\ 1 & & \end{bmatrix},
\end{equation*}  
for $M$ which lead to the sets of symplectic, pseudo-orthogonal and perplectic matrices, respectively, that play an important role in engineering and physics applications, such as Hamiltonian mechanics and optimal control problems \cite{mehrmann}. Since the particular choice of $M$ is not important for our purpose, we denote $A\in\mathbb{G}$ if $A$ satisfies \eqref{eq:automorphism} for $M\in\left\lbrace J_{2n},\Sigma_{p,q},R_n\right\rbrace$.\\
An important additional property of the square root, sign and polar we did not discuss in \Cref{sec:connection} is the fact, that the application of any of the three functions to an automorphism $A\in\mathbb{G}$ does preserve automorphism group structure \cite{structurepreserving}. That is, for an automorphism $A\in\mathbb{G}$, we have 
\begin{equation}
	\mathrm{sign}(A)\in\mathbb{G},\quad \mathcal{P}(A)\in\mathbb{G} \quad \text{and} \quad A^{1/2},A^{-1/2}\in\mathbb{G},
\end{equation} 
given that they are well defined for $A$. Thus, if we develop an iterative scheme $X_{k+1}=g(X_k)$ for evaluating any of the given functions at an automorphism $A\in\mathbb{G}$, we know that it converges to a limit $X=\lim_{k\rightarrow\infty}X_k$ that shares the automorphism structure of $A$. However, this asymptotic behaviour does in general not guarantee convergence inside the automorphism group. That is, we generally have 
\begin{equation*}
    A=X_0\in\mathbb{G}, \quad \lim_{k\rightarrow \infty}X_k=F(X_0)\in\mathbb{G},\quad \textbf{but}\quad X_k\notin \mathbb{G},
\end{equation*}
as long as $X_k$ is not sufficiently close to the limit. In \Cref{sec:pade}, we establish a class of iterative schemes that is able to preserve the automorphism structure throughout the iteration for every iterate and any of the functions. That is, given an automorphism $A\in\mathbb{G}$, we have $X_k\in\mathbb{G}$, for every $k\geq 0$. One can see that this property is not satisfied for the quadratic schemes \eqref{eq:newtonsqrt}-\eqref{eq:DBI} by taking a unitary matrix $A\in\mathrm{U}_n(\mathbb{C})$ and computing its square root using \eqref{eq:newtonsqrt}. The first Newton iterate then reads
\begin{equation*}
X_1=\frac{1}{2}(X_0+X_0^{-1}A)=\frac{1}{2}(A+A^{-1}A)=\frac{1}{2}(A+I_n),
\end{equation*} 
which is in general not unitary since
\begin{equation*}
X_1^HX_1=\frac{1}{4}(A^H+I_n)(A+I_n)=\frac{1}{4}(I_n + A^H + A + I_n)\neq I_n.
\end{equation*}
\subsubsection{The family of Padé iterations}\label{sec:pade}
A family of iterative schemes that is capable of preserving the structure is the family of Padé iterations, which was first introduced by Kenney and Laub \cite{rationalsign} for the matrix sign. Recall that a rational approximation $r_{\ell m}=\frac{p_{\ell}(x)}{q_{m}(x)}$ to a scalar function $f(x)$ is called an [$\ell /m$] Padé approximant \cite[Sec. 4.4]{matrixfunctions}, if $p_{\ell}$ and $q_{m}$ are polynomials of degree $\ell$ and $m$, respectively, $q_{m}(0)=1$ and 
\begin{equation*}
f(x) - r_{\ell m}(x)=\mathcal{O}\left(x^{\ell+m+1}\right).
\end{equation*}  
For the sign, the use of Padé approximants is motivated by the fact that, in the scalar case, \eqref{eq:sqrtidentities} can be rewritten as
\begin{equation}
\mathrm{sign}(x)=\frac{x}{(x^2)^{1/2}}=\frac{x}{\left(1-(1-x^2)\right)^{1/2}}=\frac{x}{(1-\xi)^{1/2}},
\end{equation}
where $\xi=1-x^2$. Now, based on an [$\ell /m$] Padé approximation $r_{\ell m}(\xi)=\frac{p_{\ell}(\xi)}{q_{m}(\xi)}$ to $h(\xi)=(1-\xi)^{-1/2}$, let us consider the iterative scheme
\begin{equation}
x_{k+1}=x_kr_{\ell m}(1-x_k^2)=x_k p_{\ell}(1-x_k^2) q_{m}(1-x_k^2)^{-1} \label{eq:padesignscalar}
\end{equation}
for computing the sign of a complex scalar. Kenney and Laub have given the explicit formulae for $p_{\ell}$ and $q_{m}$ for $\ell ,m\leq 3$ and proved convergence rates of \eqref{eq:padesignscalar} depending on $\ell$ and $m$ \cite[Sec. 3]{rationalsign}. They proposed to extend the iterative scheme \eqref{eq:padesignscalar} directly to matrices \cite[Sec. 5]{rationalsign}, yielding the Padé sign iteration
\begin{equation}
X_{k+1}=X_kr_{\ell m}(I_n-X_k^2),\quad X_0=A. \label{eq:padesignmatrix}
\end{equation}
Note that the evaluation of $r_{\ell m}(I_n-X^2)$ does require evaluating both $p_\ell(I_n-X^2)$ and $q_m(I_n-X^2)$ and solving the matrix equation 
\begin{equation}
r_{\ell m}(I_n-X^2)q_{m}(I_n-X^2) =p_{\ell}(I_n-X^2) \label{eq:computepade}
\end{equation}
afterwards. The evaluation of polynomials and rational functions is a field of research and a variety of algorithms for this task have been established \cite[Sec. 4]{matrixfunctions}.\\   
The Padé scheme \eqref{eq:padesignmatrix} has three main advantages over the quadratic iteration \eqref{eq:newtonsign}:
\begin{enumerate}
\item For $\ell =m$ or $\ell =m-1$, the Padé scheme \eqref{eq:padesignmatrix} converges with order of convergence $l+m+1$ \cite{rationalsign}.
\item The Padé approximant $r_{\ell m}(X_k)$ can be evaluated efficiently without explicitly computing $p_{\ell}(X_k)$ and $q_{m}(X_k)$ in every step \cite[Sec. 4.4]{matrixfunctions}.
\item The Padé scheme is structure preserving for $\ell =m$ and $m\geq1$. That is, given $A\in\mathbb{G}$, we have $X_k\in\mathbb{G}$, for all $k\geq0$ \cite{signpolar}.
\end{enumerate}
Using the results from \Cref{sec:connection}, one can obtain a Padé iteration that is similar to \eqref{eq:DBI} in the sense that its iterates converge to $A^{1/2}$ and $A^{-1/2}$ simultaneously. This can be done by applying \eqref{eq:padesignmatrix} to the special block matrix \eqref{eq:DBIinit}. Evaluating the blocks seperately yields the Padé square root iteration \cite[Thm. 4.5]{structurepreserving}
\begin{equation}
\begin{array}{lll}
Y_{k+1}=Y_kr_{\ell m}(I-Z_kY_k), &Y_0=A,\\
Z_{k+1}=r_{\ell m}(I-Z_kY_k)Z_k, &Z_0=I,
\end{array}
\label{eq:padesqrt}
\end{equation}
where $Y_k\rightarrow A^{1/2}$ and $Z_k\rightarrow A^{-1/2}$. Using $X_k^HX_k$ instead of $X_k^2$ in \eqref{eq:padesignmatrix}, one obtains the Padé iteration
\begin{equation}
X_{k+1}=X_kr_{\ell m}(I_n-X_k^HX_k),\quad X_0=A, \label{eq:padepolar}
\end{equation}
that converges to $Q=\mathcal{P}(A)$ and shares the same properties concerning convergence speed and structure preservation \cite[Thm. 2.2]{signpolar}. Note that \eqref{eq:padepolar} is applicable to both square and rectangular matrices as only $X_k^HX_k$ is required to be quadratic.\\
The most important results on the family of Padé iterations are summarized in the following Lemma:
\begin{lemma}[Padé iterations]\label{lem:pade}
Let $A\in\mathbb{C}^{n,n}$ be an automorphism, $A\in\mathbb{G}$. Then, for $\ell=m$ and $\ell\geq 1$, we have:
\begin{itemize}
\item[(i)] If $A$ has no purely imaginary eigenvalues, then $\lim_{k\rightarrow\infty} X_k=\mathrm{sign}(A)$ in \eqref{eq:padesignmatrix} and $X_k\in\mathbb{G}$, for all $k\geq0$.
\item[(ii)] If $A$ has no eigenvalues on $\mathbb{R}^-$, then 
$$\lim_{k\rightarrow\infty} \begin{bmatrix}Y_k\\Z_k\end{bmatrix}=\begin{bmatrix}A^{1/2}\\A^{-1/2}\end{bmatrix}$$ in \eqref{eq:padesqrt} and $Y_k,Z_k\in\mathbb{G}$, for all $k\geq0$.
\item[(iii)] If $A$ is nonsingular, then $\lim_{k\rightarrow\infty} X_k=\mathcal{P}(A)$ in \eqref{eq:padepolar} and $X_k\in\mathbb{G}$, for all $k\geq0$.
\end{itemize} 
Furthermore, all iterations converge with order of convergence $2\ell+1$.
\end{lemma}
\begin{proof}
See \cite[Thm. 2.2]{signpolar} for (i),(iii) and \cite[Thm. 4.5]{structurepreserving} for (ii).
\end{proof}
\subsection{Computing the Fréchet derivative of the square root, sign and polar function}
In some applications, one is not only interested in evaluating a matrix function $F$ at a given point $X$, but also in computing the Fréchet derivative $L_F(X,E)$ at that point in a given direction. For this purpose, one is in need of an appropriate procedure for simultaneously computing $F(X)$ and $L_F(X,E)$ in an efficient way. In this section, we want to introduce a direct approach based on the connections between the sign, polar and square root function that we observed in \Cref{sec:connection} and an iterative approach that is based on the Newton- and Padé-schemes introduced in \Cref{sec:quadnewton} and \Cref{sec:pade}. 
\subsubsection{Direct approaches}\label{sec:direct}
For the square root function, it is well known that the Fréchet derivative at $A$ in direction $E$ can be computed by solving the Sylvester equation
\begin{equation}
	A^{1/2}X + XA^{1/2}=E \label{eq:frechetsqrt}
\end{equation}
for $X$ \cite{sqrtsecondfrechet,matrixfunctions}. We now try to present a similar direct approach for computing the Fréchet derivatives of the sign and polar function using \eqref{eq:sqrtidentities} and the common differentiation rules. Note that we use the transpose instead of the conjugate transpose to derive the direct approach for the polar function again since the conjugate transpose is not Fréchet differentiable. However, the formula obtained from our straightforward computations can be modified to make it applicable to complex matrices using the conjugate transpose instead of the transpose in what follows \cite[Sec. 5]{polarfrechet}.\\
Recalling \eqref{eq:sqrtidentities} and applying the product rule, one can see that the Fréchet derivatives of the sign and polar function have the form 
\begin{equation*}
\begin{array}{lll}
L_{\mathrm{sign}}(A,E)&=E\left(A^2\right)^{-1/2}&+AL_{(X^2)^{-1/2}}(A,E),\\
L_{\mathcal{P}}(A,E)&=E\left(A^TA\right)^{-1/2}&+AL_{(X^TX)^{-1/2}}(A,E).
\end{array}
\end{equation*}
As a consequence, the main task is to compute the deri-vative of the composed function $f(X)=(X^2)^{-1/2}$ for the sign or $g(X)=(X^TX)^{-1/2}$ for the polar function, which again leads to solving a Sylvester equation of the type
\begin{equation}
B(A)^{1/2}X+XB(A)^{1/2}=-C(A,E),\label{eq:frechetsylvester}
\end{equation}
where the functions
\begin{equation*}
B(A)=\begin{cases}\left(A^2\right)^{-1}, & \mathrm{sign} \\
			\left(A^TA\right)^{-1}, & \mathcal{P} \end{cases}
\end{equation*} 
and 
\begin{equation*}
C(A,E)=\begin{cases}B(A)(AE+EA)B(A), & \mathrm{sign}\\
			B(A)\left(A^TE + E^TA\right)B(A), & \mathcal{P} \end{cases}
\end{equation*} 
can be obtained by applying the chain rule to $f(X)$ and $g(X)$. This method can be used to derive formulae for higher derivatives by repeatedly differentiating \eqref{eq:sqrtidentities}. However, the amount of Sylvester equations to solve increases significantly (about $2^{k-1}$ equations for the $k$-th derivative, if different directions $E_1,\dots,E_k$ are used), making it a costly approach. 
\subsubsection{Coupled Iterations}\label{sec:coupled}
Since the direct approach introduced before is fairly expensive for large matrices and the Fréchet derivative is usually not required up to full accuracy in most applications, we will now discuss an inexact, iterative procedure for evaluating a matrix function and its Fréchet derivative simultaneously.\\
It is due to Gawlik and Leok \cite{polarfrechet} that we have a comprehensive framework that works for all three functions and every iterative scheme introduced in this article. The main idea is the following: Suppose we have a Fréchet differentiable matrix function $F:\mathbb{C}^{m,n}\rightarrow\mathbb{C}^{m,n}$ and an iterative scheme $X_{k+1}=g(X_k)$ induced by a Fréchet differentiable update function $g:\mathbb{C}^{m,n}\rightarrow \mathbb{C}^{m,n}$ which we can use for evaluating $F$ at a given point $A$. Then, if the Fréchet derivative $L_g(X,E)$ is known explicitly, we can consider the coupled iteration 
\begin{equation}
\begin{array}{ll}
X_{k+1}=g(X_k), &X_0=A,\\
E_{k+1}=L_g(X_k,E_k), &E_0=E,
\end{array}\label{eq:generalcoupled}
\end{equation}
to obtain iterates $X_k$ and $E_k$ simultaneously. The coupled scheme \eqref{eq:generalcoupled} can be shown to be convergent to $F(A)$ and $L_F(A,E)$, respectively, if differentiation and limit formation do commute for $g$ \cite[Sec. 1]{polarfrechet}.\\
For the quadratic iterations \eqref{eq:newtonsqrt}, \eqref{eq:newtonsign}, \eqref{eq:newtonpolarsquare} and \eqref{eq:newtonpolar}, this property is fulfilled and one can obtain the iteration
\begin{equation}
\begin{array}{ll}
X_{k+1} = \frac{1}{2}\left(X_k + X_k^{-1}\right), &X_0=A,\\
E_{k+1} = \frac{1}{2}\left(E_k - X_k^{-1}E_kX_k^{-1}\right), &E_0=E,
\end{array} \label{eq:newtonfrechetsign}
\end{equation}
for the matrix sign \cite[Thm. 5.6]{matrixfunctions} and similar schemes for the square root iterations \eqref{eq:newtonsqrt}, \eqref{eq:DBI} and the real polar iterations \eqref{eq:newtonpolarsquarereal}, \eqref{eq:newtonpolarreal}. We point out that $A$ has to be considered as $X_0$ and differentiated in direction $E=E_0$ in \eqref{eq:newtonsqrt} to obtain a suitable iteration for $A^{1/2}$ and $L_{X^{1/2}}(A,E)$. The complex polar iterations \eqref{eq:newtonpolarsquare} and \eqref{eq:newtonpolar} again have to be considered seperately, which will be briefly discussed at the end of this section.\\
Looking again at the Padé family of iterations, one can verify that for a Padé approximant $h(X)=r_{\ell m}(I-X^2)=p_{\ell}(I-X^2)q_{m}(I-X^2)^{-1}$, $h$ can be differentiated using the chain and product rule to obtain
\begin{eqnarray}
L_h(X,E)&=&L_{r_{\ell m}}\left(I-X^2,-(XE+EX)\right)\label{eq:frechetpade}\\
&=& \begin{aligned}[t] &L_{p_{\ell}}\left(I-X^2,-(XE+EX)\right)q_{m}(I-X^2)^{-1}\\
		&-r_{\ell m}(I-X^2)L_{q_{m}}\left(I-X^2,-(XE+EX)\right)\\
		& \cdot q_{m}(I-X^2)^{-1}.\end{aligned}\notag 
\end{eqnarray}
Using this result, the coupled Padé-iteration 
\begin{equation}
\begin{array}{ll}
X_{k+1} = X_k h(X_k), &X_0=A,\\
E_{k+1} = E_k h(X_k) + X_k L_h(X_k,E_k), &E_0=E,
\end{array} \label{eq:newtonfrechetpadesign}
\end{equation}
converges to $\mathrm{sign}(A)$ and $L_{\mathrm{sign}}(A,E)$, respectively. Similarly, using $\widetilde{h}(X)=r_{\ell m}(I_n - X^TX)$ instead of $h$ and forming the derivative in \eqref{eq:frechetpade} accordingly yields coupled iterations for the polar mapping $\mathcal{P}(A)$  of a real matrix $A$ and its derivative\cite[Thm. 1]{polarfrechet}. Furthermore, using \eqref{eq:newtonfrechetpadesign}, there is again a coupled iteration for the square root and inverse square root as well as their derivatives that is based on \eqref{eq:padesqrt}.\\
Now recall the iterations involving the conjugate transpose \eqref{eq:newtonpolarsquare}, \eqref{eq:newtonpolar} and \eqref{eq:padepolar} for the computation of $\mathcal{P}(A)$ for complex matrices $A$. As mentioned before, the mapping $X\mapsto X^H$ is not Fréchet differentiable, thus making the update function $g$ not Fréchet differentiable for any of the three iterative schemes. However, it can be shown that coupled iterative schemes exist in the case of complex matrices and can be obtained by taking the real iterations and replacing every transpose by a conjugate transpose. For the quadratic schemes resulting from \eqref{eq:newtonpolarsquare} and \eqref{eq:newtonpolar}, convergence can be proven for any $A$ having full rank, while in the case of the Padé schemes \eqref{eq:padepolar}, convergence can be shown under additional assumptions \cite{polarfrechet}.\\
For the remainder of this paper, our further investigation requires all matrices $A$ and $E$ to be real. As a consequence, it suffices to consider the polar-iterations involving the transpose from now on. 
\section{The Complex Step}
\subsection{Motivation and advantages}
In most practical applications, the function of interest is either too difficult to differentiate by hand or there is not even a closed form representation for $f$ as only its effect $f(A)$ is given for a sample of points $A$. In those cases, being able to approximate the function $f$ at a given point $A$ and its (Fréchet) derivative $f'(A)$ ($L_f(A,E)$, respectively) efficiently and precisely is a major task in numerical analysis \cite{applic1, applic3, applic2}.\\
The easiest approach is to first evaluate $f$ at the desired point $A\in\mathbb{R}^{m,n}$ and then use the finite difference approximation 
\begin{equation}
L_f(A,E)\approx\begin{cases}\frac{f(A+hE)-f(A)}{h} &\text{(forward differences)}\\
							\frac{f(A+hE)-f(A-hE)}{2h} &\text{(central differences)}\end{cases} \label{eq:finite}
\end{equation}
to approximate the corresponding Fréchet derivative in direction of $E\in\mathbb{R}^{m,n}$. The main drawback of this approach is that \eqref{eq:finite} suffers from cancellation errors due to subtraction in the numerator as $h$ tends to zero. As a consequence, $h$ can usually not be chosen siginificantly smaller than the square root of the unit roundoff and the approximation quality is thus limited by the error that is achieved choosing $h$ numerically optimal \cite{finitedifference2, finitedifference1}.\\
A different approach to numerically approximate derivatives, that was first used by Lyness and Moler \cite{CSmoler}, is to exploit complex arithmetics for $f$ to avoid cancellation. The idea of this so-called \textit{Complex Step approach} (CS) is to inspect the complex Taylor expansion  
\begin{equation}
f(x+ih)=f(x)+ihf'(x)-\frac{h^2}{2}f''(x)-\frac{ih^3}{6}f'''(x)+\mathcal{O}(h^4)
\end{equation}
for an analytic function $f$ and observe, that splitting real and imaginary parts yields the approximations 
\begin{align}
f(x)&= \mathrm{Re}\left(f(x+ih)\right)+\mathcal{O}(h^2), \notag\\
f'(x)&=\mathrm{Im}\left(\frac{f(x+ih)}{h}\right)+\mathcal{O}(h^2), \label{eq:CSscalar}
\end{align}
that do not suffer from numerical issues. This idea can now be directly transferred to Fréchet-differentiable matrix functions given a direction matrix $E$ to obtain the approximation \cite{complexstep}
\begin{equation}
L_f(A,E)=\mathrm{Im}\left(\frac{f(A+ihE)}{h}\right)+\mathcal{O}(h^2).\label{eq:CSmatrix}
\end{equation}
Unlike in \eqref{eq:finite}, $h$ can be chosen as small as needed to achieve the desired approximation accuracy in \eqref{eq:CSscalar} and \eqref{eq:CSmatrix}, even making choices like $h=10^{-100}$ applicable if needed \cite{CSNPL}.  
Nonetheless, there are three difficulties to be aware of.\\
First of all, an arbitrarily small choice of $h$ is only possible, if $A$ and $E$ are real, $f$ is real if evaluated at a real argument and no complex arithmetic is exploited to evaluate $f$. Else, the CS behaves similar to other second order methods, such as the central difference approximation from \eqref{eq:finite} in the sense that $h$ can usually not be chosen smaller than the square root of the unit roundoff \cite[Fig. 7.1]{complexstep}.\\
Second, evaluating $f$ at the complex argument $A+ihE$ might be computationally expensive compared to evaluating $f$ at a real argument twice ($f(A+hE)$ and $f(A)$ or $f(A-hE)$) since complex operations are more costly than real operations. However, seperating the real and imaginary parts for the operations, there are possibilities for parallelization \cite{complexmatrix, complexinverse}.\\
The third point is that the CS approximation, in theory, requires $f$ to be an analytic function to be able to argue based on the complex Taylor expansion. However, this property was shown to be sufficient, but not necessary, as we will see in subsequent sections \cite{complexstep}. Considering the three functions introduced in \Cref{sec:functions}, only the square root (and its inverse) are analytical and, as such, compatible with the CS. For the sign and polar function, we will show in the following chapter that the CS is applicable as well.
\subsection{Combining the CS and Newton's method}\label{sec:newtonCS}
The approach to use the Complex Step in combination with Newtons method was first proposed by Al-Mohy and Higham\cite{complexstep} for the matrix sign using the quadratic Newton sign iteration \eqref{eq:newtonsign}. The main idea is to explore the convergence behaviour of the complex iteration
\begin{equation}
\widehat{X}_{k+1}=g\left(\widehat{X}_k\right),\quad \widehat{X}_0 = A+ihE, \label{eq:CSgeneral}
\end{equation}
for evaluating the non-analytic function $f$ at $\widehat{X}_0$, given an analytic update function $g(X)$ and real matrices $A$ and $E$.\\ 
Al-Mohy and Higham observed the following result for the matrix sign function \cite[Thm. 5.1]{complexstep}:
\begin{theorem}[The CS for the sign function]\label{thm:CSsign}
Let $A,E\in\mathbb{R}^{n,n}$ with $A$ having no purely imaginary eigenvalues. Consider the CS-iteration (cf. \eqref{eq:newtonsign})
\begin{equation}
\widehat{X}_{k+1}=\frac{1}{2}\left(\widehat{X}_k+\widehat{X}_k^{-1}\right),\quad \widehat{X}_0=A+ihE. \label{eq:CSsign} 
\end{equation}
Then, for any $h$ being sufficiently small, the iterates $\widehat{X}_k\in\mathbb{C}^{n,n}$ are nonsingular. Moreover, the following hold:
\begin{align}
&\mathrm{Re}\left(\mathrm{sign}(A+ihE)\right)=\lim_{k\rightarrow\infty}\mathrm{Re}\left(\widehat{X}_k\right)=\mathrm{sign}(A)+\mathcal{O}(h^2)\notag\\
&\mathrm{Im}\Biggl(\frac{\mathrm{sign}(A+ihE)}{h}\Biggr)=\lim_{k\rightarrow\infty}\mathrm{Im}\Biggl(\frac{\widehat{X}_k}{h}\Biggr)=L_{\mathrm{sign}}(A,E)+\mathcal{O}(h^2)\label{eq:CSnonanalytic}
\end{align}
\end{theorem}
\noindent 
\Cref{thm:CSsign} and \eqref{eq:CSnonanalytic} in particular show, that it is not necessary to have an analytic function for the CS to work. For us, the fact that the CS is applicable to the matrix sign function if done correctly is of special interest as well. To state and prove these observations, it suffices to analyze the asymptotic behaviour of the quadratic CS-sign-iteration. However, in their proof, Al-Mohy and Higham show that for the CS-iterates $\widehat{X}_k$, the more general result  
\begin{equation}
\mathrm{Re}\left(\widehat{X}_k\right)=X_k+\mathcal{O}(h^2)\text{ }\text{  and  }\text{ }\mathrm{Im}\left(\widehat{X}_k\right)=hE_k + \mathcal{O}(h^3) \label{eq:strongIV}
\end{equation}
holds for every $k\geq 0$, where $X_k$ and $E_k$ are the iterates generated by the coupled Newton iteration \eqref{eq:newtonfrechetsign} initialized with $X_0=A$ and $E_0=E$. Since this coupled scheme is known to converge to $\mathrm{sign}(A)$ and $L_{\mathrm{sign}}(A,E)$, respectively, the asymptotic convergence of \eqref{eq:CSgeneral} follows immediately for suitable $h$ \cite{matrixfunctions}.\\
Now taking a closer look at \eqref{eq:strongIV}, this property in fact reveals equality of the iterates from the coupled iteration \eqref{eq:newtonfrechetsign} and the corresponding CS iteration \eqref{eq:CSgeneral}, if one seperates the real and imaginary parts (up to order of $h^2$). The benefit of \eqref{eq:strongIV} over \eqref{eq:CSnonanalytic} becomes clear when we consider iterative schemes with additional features, such as structure preservation properties, since \eqref{eq:strongIV} states that this property is approximately inherited by the CS-iteration (up to order $h^2$). This topic will be discussed for the family of Padé iterations in \Cref{sec:CSpade}.\\
Before we get to the Padé iterations, let us extend the CS idea to the quadratic polar iteration \eqref{eq:newtonpolarreal}, the square root iteration \eqref{eq:newtonsqrt} and the DB-iteration \eqref{eq:DBI}.
\begin{theorem}[The CS for the Polar function] \label{thm:CSpolar}
Let $A,E\in\mathbb{R}^{m,n}$ with $\mathrm{rank}(A)=n\geq m$. Consider the CS-iteration (cf. \eqref{eq:newtonpolarreal})
\begin{equation}
\widehat{X}_{k+1}=\frac{1}{2}\widehat{X}_k\left(I_n + \left(\widehat{X}_k^T\widehat{X}_k\right)^{-1}\right),\quad \widehat{X}_0=A+ihE. \label{eq:CSpolar}
\end{equation}
Then, for any $h$ being sufficiently small, the iterates $\widehat{X}_k\in\mathbb{C}^{m,n}$ have full rank. Moreover, the following hold:
\begin{align*}
&\mathrm{Re}\left(\mathcal{P}(A+ihE)\right)=\lim_{k\rightarrow\infty}\mathrm{Re}\left(\widehat{X}_k\right)=\mathcal{P}(A)+\mathcal{O}(h^2)\tag{i}\\
&\mathrm{Im}\left(\frac{\mathcal{P}(A+ihE)}{h}\right)=\lim_{k\rightarrow\infty}\mathrm{Im}\left(\frac{\widehat{X}_k}{h}\right)=L_{\mathcal{P}}(A,E)+\mathcal{O}(h^2) \tag{ii}
\end{align*}
\end{theorem}
\begin{proof}
The proof works in the same way as \Cref{thm:CSsign}, where \eqref{eq:strongIV} is shown for \eqref{eq:CSpolar} and the coupled Fréchet iteration 
\begin{equation*}
\begin{array}{ll}
X_{k+1} = \frac{1}{2}X_k\left(I + (X_k^TX_k)^{-1}\right), &X_0=A,\\
E_{k+1} = \frac{1}{2}\Bigl[E_k(I+(X_k^TX_k)^{-1}) 	   & \\			
		- X_k(X_k^TX_k)^{-1}(X_k^TE_k+E_k^TX_k)(X_k^TX_k)^{-1}\Bigr], &E_0=E,
\end{array}
\end{equation*} 
obtained from \eqref{eq:newtonpolar}. Since $X_0$ has full rank, $h$ can be chosen small enough to achieve full rank of $\widehat{X}_0$ as well. Since $\widehat{X}_0^T\widehat{X}_0$ is then symmetric and positive definite (s.p.d.), $I_n + (\widehat{X}_0^T\widehat{X}_0)^{-1}$ is s.p.d. as well and $\widehat{X}_1$ has full rank. Inductively, $\widehat{X}_k$ can be shown to have full rank as well, for every $k$.
\end{proof}
\begin{theorem}[The CS for the Square root function]\label{thm:CSsqrt}
Let $A,E\in\mathbb{R}^{n,n}$ with $A$ having no eigenvalues on $\mathbb{R}^-$. 
\begin{enumerate}
\item Consider the CS-iteration (cf. \eqref{eq:newtonsqrt})
\begin{equation}
\widehat{X}_{k+1}=\frac{1}{2}\left(\widehat{X}_k+ \widehat{X}_k^{-1}\widehat{X}_0\right),\quad \widehat{X}_0=A+ihE.
\end{equation}
Then, for any $h$ being sufficiently small, the iterates $\widehat{X}_k\in\mathbb{C}^{n,n}$ are nonsingular. Moreover, the following hold:
\begin{align*}
&\lim_{k\rightarrow\infty}\mathrm{Re}\left(\widehat{X}_k\right)=A^{1/2}+\mathcal{O}(h^2) \tag{i}\\
&\lim_{k\rightarrow\infty}\mathrm{Im}\left(\frac{\widehat{X}_k}{h}\right)=L_{X^{1/2}}(A,E)+\mathcal{O}(h^2) \tag{ii}
\end{align*}
\item Consider the CS-DB-iteration (cf. \eqref{eq:DBI})
\begin{equation}
\begin{array}{ll}
\widehat{Y}_{k+1}=\frac{1}{2}\left(\widehat{Y}_k+ \widehat{Z}_k^{-1}\right),		&\widehat{Y}_0=A+ihE,\\
\widehat{Z}_{k+1}=\frac{1}{2}\left(\widehat{Z}_k+ \widehat{Y}_k^{-1}\right),		&\widehat{Z}_0=I_n.
\end{array} \label{eq:CSDBI}
\end{equation}
Then, for any $h$ being sufficiently small, the iterates $\widehat{Y}_k\in\mathbb{C}^{n,n}$ and $\widehat{Z}_k\in\mathbb{C}^{n,n}$ are nonsingular. Moreover, the following hold:
\begin{align*}
&\lim_{k\rightarrow\infty}\mathrm{Re}\left(\widehat{Y}_k\right)=A^{1/2}+\mathcal{O}(h^2)\tag{i}\\
&\lim_{k\rightarrow\infty}\mathrm{Re}\left(\widehat{Z}_k\right)=A^{-1/2}+\mathcal{O}(h^2)\tag{ii}\\
&\lim_{k\rightarrow\infty}\mathrm{Im}\left(\frac{\widehat{Y}_k}{h}\right)=L_{X^{1/2}}(A,E)+\mathcal{O}(h^2)\tag{iii}\\
&\lim_{k\rightarrow\infty}\mathrm{Im}\left(\frac{\widehat{Z}_k}{h}\right)=L_{X^{-1/2}}(A,E)+\mathcal{O}(h^2)\tag{iv}
\end{align*}
\end{enumerate} 
\end{theorem}
\begin{proof}
The proof of 1.) again works in the same way as \Cref{thm:CSsign} and \Cref{thm:CSpolar}.\\
For 2.), we use the fact that the DB-iteration can be obtained by using the sign iteration \eqref{eq:newtonsign} for $X_0=\begin{bmatrix}0&A\\I & 0\end{bmatrix}$. Thus, \eqref{eq:CSDBI} can be obtained from the CS-sign-iteration using 
\begin{equation}
\widehat{X}_0=X_0+ihE_0=X_0+ih\begin{bmatrix}0 & E\\ 0 & 0\end{bmatrix}=\begin{bmatrix}0 & A+ihE\\ I & 0\end{bmatrix} \label{eq:CSdbiinit}
\end{equation}
as the initial value. Since $\widehat{X}_k$ is nonsingular by \Cref{thm:CSsign}, $\widehat{Y}_k$ and $\widehat{Z}_k$ have to be nonsingular as the (1,2)- and (2,1)-block of $\widehat{X}_k$.
\end{proof}
\noindent
Note that the CS works for the square root directly as it is an analytic function, so the results from \Cref{thm:CSsqrt} are to be expected. Nevertheless, it might be beneficial to use iterative methods over direct methods for computing the square root if $A$ has some additional properties, such as sparsity or group structure, which is of particular interest in the following.
\section{The Complex Step for the family of Padé iterations}\label{sec:CSpade}
In this section, we try to extend the results obtained in \Cref{sec:newtonCS} to the family of Padé iterations. Recall that for the sign and polar function, the Padé iterative schemes are induced by the update function
\begin{eqnarray*}
g(X)&=&\begin{cases}Xp_{\ell}(I-X^2)q_{m}(I-X^2)^{-1} & \mathrm{sign}\\
				  Xp_{\ell}(I-X^TX)q_{m}(I-X^TX)^{-1} & \mathcal{P}\end{cases}\\
	&=&\begin{cases}Xr_{\ell m}(I-X^2)& \mathrm{sign}\\
				  Xr_{\ell m}(I-X^TX)& \mathcal{P}\end{cases}
\end{eqnarray*}
In particular, $g$ is a rational function in both cases and thus sufficiently smooth for the CS-approach to be applicable to $g$. The following theorem presents our main result concerning the CS and the family of Padé iterations:
\begin{theorem}[The CS for the family of Padé iterations]\label{thm:CSpade}
Let $\ell\geq 1$, $A,E\in\mathbb{R}^{m,n}$.
\begin{enumerate}
\item If $m=n$ and $A$ has no purely imaginary eigenvalues, consider the CS-Padé-sign-iteration (cf. \eqref{eq:padesignmatrix})
\begin{equation}
	\widehat{X}_{k+1}=\widehat{X}_k r_{\ell\ell}\left(I-\widehat{X}_k^2\right),\quad \widehat{X}_0=A+ihE. \label{eq:CSpadesignthm}
\end{equation}
Then, for any $h$ being sufficiently small, the iterates $\widehat{X}_k\in\mathbb{C}^{n,n}$ are nonsingular. Moreover, the following hold:
\begin{align*}
&\mathrm{Re}\left(\widehat{X}_k\right)=X_k + \mathcal{O}(h^2) \rightarrow \mathrm{sign}(A)+\mathcal{O}(h^2)\tag{i}\\
&\mathrm{Im}\left(\frac{\widehat{X}_k}{h}\right)=E_k+\mathcal{O}(h^2) \rightarrow L_{\mathrm{sign}}(A,E)+\mathcal{O}(h^2) \tag{ii}
\end{align*}
Here, $X_k$ and $E_k$ are the iterates generated by the iterative scheme \eqref{eq:newtonfrechetpadesign} converging to $\mathrm{sign}(A)$ and $L_{\mathrm{sign}}(A,E)$, respectively. 
\item If $m\geq n$ and $A$ has full rank, consider the CS-Padé-Polar-iteration (cf. \eqref{eq:padepolar} using the transpose)
\begin{equation}
	\widehat{X}_{k+1}=\widehat{X}_k r_{\ell\ell}\left(I-\widehat{X}_k^T\widehat{X}_k\right),\quad \widehat{X}_0=A+ihE. \label{eq:CSpadepolarthm}
\end{equation}
Then, for any $h$ being sufficiently small, the iterates $\widehat{X}_k\in\mathbb{C}^{m,n}$ have full rank. Moreover, the following hold:
\begin{align*}
&\mathrm{Re}\left(\widehat{X}_k\right)=X_k + \mathcal{O}(h^2) \rightarrow \mathcal{P}(A)+\mathcal{O}(h^2)\tag{i}\\
&\mathrm{Im}\left(\frac{\widehat{X}_k}{h}\right)=E_k+\mathcal{O}(h^2) \rightarrow L_{\mathcal{P}}(A,E)+\mathcal{O}(h^2) \tag{ii}
\end{align*}
Here, $X_k$ and $E_k$ are the iterates generated by the coupled scheme obtained from \eqref{eq:padepolar} (see \eqref{eq:frechetpade},\eqref{eq:newtonfrechetpadesign} for $\widehat{h}(X)=r_{\ell \ell}(I-X^TX)$) converging to $\mathcal{P}(A)$ and $L_{\mathcal{P}}(A,E)$, respectively. 
\end{enumerate}
Additionaly, if $A\in\mathbb{R}^{n,n}$ is an automorphism ($A\in\mathbb{G}$), the structure is approximately preserved during both iterations. That is, we have
\begin{equation}
\mathrm{dist}\left(\mathrm{Re}\left(\widehat{X}_k\right),\mathbb{G}\right)=\min_{G\in\mathbb{G}}||\mathrm{Re}\left(\widehat{X}_k\right)-G||=\mathcal{O}(h^2), \label{eq:distance}
\end{equation}
for every $k\geq 0$.
\end{theorem} 
\begin{proof}
We denote $s(X)=X^2$ and $t(X)=X^TX$ and do the proof for square $A$ and $E$ only. In the case of a full rank rectangular $A$, $t(A)$ is square and nonsingular and similar arguments hold to perform an analogous proof.\\
We will show the identity 
\begin{equation}
	\widehat{X}_k = X_k + ihE_k + \mathcal{O}(h^2) + i\mathcal{O}(h^3),\text{ }k\geq 0 \label{eq:IV}
\end{equation}
by induction, where $X_k,E_k$ are the iterates generated by the coupled scheme \eqref{eq:newtonfrechetpadesign} or the corresponding scheme for $(\mathcal{P}(A),L_{\mathcal{P}}(A,E))$ using $\widetilde{h}(X)=r_{\ell\ell}(I_n-X^TX)$, respectively. The initial matrix $A$ is assumed to satisfy the conditions stated in 1. and 2. for the desired function. Since $X_k$ and $E_k$ are known to converge \cite{polarfrechet}, showing \eqref{eq:IV} implies (i) and (ii). The approximate structure preservation stated in \eqref{eq:distance} then follows immediately from \eqref{eq:IV}.\\
For $k=0$, we set $\widehat{X}_0=A+ihE$, choose $f(X)$ to be either $s(X)$ or $t(X)$ and compute the first CS-iterate
\begin{eqnarray*}
\widehat{X}_1&=&\widehat{X}_0r_{\ell\ell}\left(I-f\left(\widehat{X}_0\right)\right)\\
			 &=& (A+ihE) r_{\ell\ell}\left(I-f(A+ihE)\right).
\end{eqnarray*}
Since $f$ and $r_{\ell\ell}$ are Fréchet differentiable, we have 
\begin{equation}
	f(A+ihE)=f(A)+ihL_f(A,E)+\mathcal{O}(h^2)+i\mathcal{O}(h^3) \label{eq:ftaylor}
\end{equation} 
and
\begin{eqnarray}
	r_{\ell\ell}\left(I-f(A+ihE)\right)&\stackrel{\eqref{eq:ftaylor}}{=}&r_{\ell\ell}\left(I-f(A)+ihL_f(A,E)+\mathcal{O}(h^2)+i\mathcal{O}(h^3)\right)\notag\\
					&=&r_{\ell\ell}\left(I-f(A)\right)+ihL_{r_{\ell\ell}}\left(I-f(A),L_f(A,E)\right)\notag\\
					&{}&+ \mathcal{O}(h^2)+i\mathcal{O}(h^3) \label{eq:rtaylor}
\end{eqnarray}
using Taylors formula. In consequence, we obtain the first CS-iterate
\begin{eqnarray*}
	\widehat{X}_1 &=& (A+ihE) r_{\ell\ell}\left(I-f(A+ihE)\right)\\
				  &\stackrel{\eqref{eq:rtaylor}}{=}& \begin{aligned}[t] &(A+ihE) \Bigl[r_{\ell\ell}\left(I-f(A)\right) + \mathcal{O}(h^2)\\
				  & +ihL_{r_{\ell\ell}}\left(I-f(A),L_f(A,E)\right) +i\mathcal{O}(h^3)\Bigr] \end{aligned}\\
				  &=& \begin{aligned}[t] &A r_{\ell\ell}\left(I-f(A)\right) + ih\Bigl[E r_{\ell\ell}\left(I-f(A)\right) \\  &+ A L_{r_{\ell\ell}}\left(I-f(A),L_f(A,E)\right)\Bigr]+ \mathcal{O}(h^2)+i\mathcal{O}(h^3)\end{aligned}\\
				  &\stackrel{\eqref{eq:newtonfrechetpadesign}}{=}& X_1 + \mathcal{O}(h^2) + ih\left(E_1 + \mathcal{O}(h^2)\right).  		
\end{eqnarray*}
Since $X_1$ is obtained from the standard Padé-scheme \eqref{eq:padesignmatrix} or \eqref{eq:padepolar}, it again satisfies 1. or 2., respectively.\\
Due to the relation
\begin{equation*}
	\widehat{X}_1=X_1 + \mathcal{O}(h)
\end{equation*}
and $\mathrm{sign}(X)$ and $\mathcal{P}(X)$ being continuous on the given subsets (see the thoughts in \Cref{rem:signpolarsqrt}), $\widehat{X}_1$ satisfies 1. or 2. as well for $h$ being sufficiently small.\\
For $k\geq 0$, let $\widehat{X}_k = X_k + ihE_k + \mathcal{O}(h^2) + i\mathcal{O}(h^3)$. Then, we have 
\begin{eqnarray*}
	f(\widehat{X}_k)&\stackrel{\eqref{eq:IV}}{=}&f\left(X_k+ihE_k+\mathcal{O}(h^2)+i\mathcal{O}(h^3)\right)\\
					&=&f(X_k) + ih L_f(X_k,E_k) + \mathcal{O}(h^2)+i\mathcal{O}(h^3)
\end{eqnarray*} 
and
\begin{eqnarray*}
	r_{\ell\ell}\left(I-f(\widehat{X}_k)\right)&=&r_{\ell\ell}\left(I-f(X_k)+ihL_f(X_k,E)+\mathcal{O}(h^2)+i\mathcal{O}(h^3)\right) \\
	&\stackrel{\eqref{eq:rtaylor}}{=}& r_{\ell\ell}\left(I-f(X_k)\right)+ihL_{r_{\ell\ell}}\left(I-f(X_k),L_f(X_k,E_k)\right)\\
	&{}&+ \mathcal{O}(h^2)+i\mathcal{O}(h^3). 
\end{eqnarray*}
Using these computations, the $(k+1)$-st CS-iterate reads
\begin{eqnarray}
	\widehat{X}_{k+1} &=& \widehat{X}_k r_{\ell\ell}\left(I-f(\widehat{X}_k)\right) \notag\\
				  & \stackrel{\eqref{eq:IV}}{=}& \begin{aligned}[t] &\left(X_k+ihE_k + \mathcal{O}(h^2)+i\mathcal{O}(h^3)\right)\\
				  &\cdot \Bigl[r_{\ell\ell}\left(I-f(X_k)\right) + \mathcal{O}(h^2)\\
				  &+ihL_{r_{\ell\ell}}\left(I-f(X_k),L_f(X_k,E_k)\right)+i\mathcal{O}(h^3)\Bigr]\end{aligned} \notag \\
				  &=& \begin{aligned}[t] &X_k r_{\ell\ell}\left(I-f(X_k)\right) + ih\Bigl[E_k r_{\ell\ell}\left(I-f(X_k)\right)  \\
				   & + X_k L_{r_{\ell\ell}}\left(I-f(X_k),L_f(X_k,E_k)\right)\Bigr]\\
				   & + \mathcal{O}(h^2)+i\mathcal{O}(h^3)\\				  
\end{aligned}\notag\\				   
				  & \stackrel{\eqref{eq:newtonfrechetpadesign}}{=}& X_{k+1} + \mathcal{O}(h^2) + ih\left(E_{k+1} + \mathcal{O}(h^2)\right),  \label{eq:inducproof}	
\end{eqnarray}
concluding the induction.\\
Taking a closer look at \eqref{eq:inducproof}, it is immediate to see that 
\begin{equation}
	\mathrm{Re}\left(\widehat{X}_k\right)=X_k+\mathcal{O}(h^2),{ }\forall k\geq 0 \label{eq:coupledidentity}
\end{equation}
and for every $k\geq 0$, we can find a suitable real matrix $B_k\in\mathbb{R}^{n,n}$ such that $\mathrm{Re}\left(\widehat{X}_k\right)=X_k+h^2B_k$. Thus, for $A$ being an automorphism, the CS-iterates $\widehat{X}_k$ satisfy
\begin{eqnarray*}
\mathrm{dist}\left(\mathrm{Re}\left(\widehat{X}_k\right),\mathbb{G}\right)&\stackrel{\eqref{eq:coupledidentity}}{=}&\min_{G\in\mathbb{G}}||X_k+h^2B_k-G||\\
																		& \leq & \min_{G\in\mathbb{G}}\left(||X_k-G|| + h^2||B_k||\right)\\
																		& \stackrel{(*)}{=} & h^2 ||B_k||,
\end{eqnarray*}
where $(*)$ follows from $X_k$ being an automorphism by \Cref{lem:pade}. Now we define $||B_k||\coloneqq C$ and conclude 
$$\mathrm{dist}\left(\mathrm{Re}\left(\widehat{X}_k\right),\mathbb{G}\right)\leq C h^2 =\mathcal{O}(h^2).$$
\end{proof}
\noindent
Due to \Cref{thm:CSpade}, we now have an approach for evaluating the sign or polar function at a given point and computing the Fréchet derivative in a given direction simultaneously, that has a few benefits compared to the approaches from \Cref{sec:coupled} and \Cref{sec:newtonCS}: First of all, it is very easy to implement, as long as one is using a programming language that is able to handle complex arithmetics. As opposed to the coupled iterations from \Cref{sec:coupled}, where computing the Fréchet derivative $L_g(X_k,E_k)$ is necessary in every Newton step, evaluating the rational function $r_{\ell\ell}(I-X^2)$ (or $r_{\ell\ell}(I-X^TX)$, respectively)  at a complex argument efficiently is the main computational effort required in the CS scenario.\\
Secondly, the parameter $\ell$ can, in theory, be chosen freely to achieve the order of convergence one wants to accomplish. However, in our experiments, we observed that choosing $\ell\geq 3$ leads to significant numerical problems, since forming higher powers of $X_k$ can lead to ill-conditioned matrices $p(I-X^2)$ and, more importantly, $q(I-X^2)$, that are used for solving 
\begin{equation*}
r_{\ell\ell}(I-X^2)q_{\ell}(I-X^2)=p_{\ell}(I-X^2)
\end{equation*}
to obtain $r_{\ell\ell}(I-X^2)$. Thus, we propose to choose $\ell=1,2$, which still leads to an order of convergence of three or five, respectively, which is superior to the quadratic order of convergence of the standard Newton methods from \Cref{sec:newtonCS}.\\
The third benefit is that one does not have to know the Fréchet derivative $L_g(X)$ explicitly, which, in some applications, might be difficult due to $g$ being a function that is demanding to differentiate or not given in a closed form.\\
Finally, choosing $h$ as small as necessary to achieve the structure preservation property of the Padé scheme does not affect the quality of the CS-approximation, as opposed to using finite differences. As a consequence, the real part of $\widehat{X}_k$ will be as close to an automorphism as one desires.\\
To conclude with the family of Padé iterations, we now present the result for the Padé square root iteration. The proof works in the same way as \Cref{thm:CSpade} or can be diretly obtained from \eqref{eq:CSpadesignthm} using $\widehat{X}_0$ as in \eqref{eq:CSdbiinit}:
\begin{theorem}[The CS for the Padé Square root iteration] \label{thm:CSpadesqrt}
Let $\ell\geq 1$, $A,E\in\mathbb{R}^{n,n}$, with $A$ having no eigenvalues in $\mathbb{R}^-$. Consider the CS-iteration (cf. \eqref{eq:padesqrt})
\begin{equation}
\begin{array}{ll}
\widehat{Y}_{k+1}=\widehat{Y}_k r_{\ell\ell}(I-\widehat{Z}_k\widehat{Y}_k),		&\widehat{Y}_0=A+ihE,\\
\widehat{Z}_{k+1}= r_{\ell\ell}(I-\widehat{Z}_k\widehat{Y}_k) \widehat{Z}_k,		&\widehat{Z}_0=I.
\end{array} \label{eq:CSpadesqrt}
\end{equation}
Then, for any $h$ being sufficiently small, the iterates $\widehat{Y}_k\in\mathbb{C}^{n,n}$ and $\widehat{Z}_k\in\mathbb{C}^{n,n}$ are nonsingular. Moreover, the following hold:
\begin{align*}
&\mathrm{Re}\left(\widehat{Y}_k\right)=Y_k+\mathcal{O}(h^2)\rightarrow A^{1/2} +\mathcal{O}(h^2)\tag{i}\\
&\mathrm{Re}\left(\widehat{Z}_k\right)=Z_k+\mathcal{O}(h^2)\rightarrow A^{-1/2}+\mathcal{O}(h^2)\tag{ii}\\
&\mathrm{Im}\left(\frac{\widehat{Y}_k}{h}\right)=D_k+\mathcal{O}(h^2)\rightarrow L_{X^{1/2}}(A,E)+\mathcal{O}(h^2)\tag{iii}\\
&\mathrm{Im}\left(\frac{\widehat{Z}_k}{h}\right)=F_k+\mathcal{O}(h^2)\rightarrow L_{X^{-1/2}}(A,E)+\mathcal{O}(h^2)\tag{iv}
\end{align*}
Here, $D_k$ and $F_k$ are the iterates generated by the coupled iterative scheme obtained from \eqref{eq:newtonfrechetpadesign} for $\tilde{h}(X)\equiv \tilde{h}(Y,Z)=r_{\ell\ell}(I-ZY)$ converging to $L_{X^{1/2}}(A,E)$ and $L_{X^{-1/2}}(A,E)$, respectively. Additionaly, if $A\in\mathbb{G}$ is an automorphism, the structure is approximately preserved for $\widehat{Y}_k$ and $\widehat{Z}_k$. That is, we have  
\begin{equation*}
\mathrm{dist}\left(\mathrm{Re}\left(\widehat{Y}_k\right),\mathbb{G}\right)=\mathcal{O}(h^2)\text{ }\text{ and }\text{ }\mathrm{dist}\left(\mathrm{Re}\left(\widehat{Z}_k\right),\mathbb{G}\right)=\mathcal{O}(h^2).
\end{equation*}
\end{theorem} 
\subsection{A note on higher derivatives}
We want to note, that the CS can be used in combination with the coupled iterations discussed in \Cref{sec:coupled} to obtain iterative CS-schemes for higher derivatives. For the second derivative, assume that $g$ is at least twice differentiable and compatible with the CS and consider the coupled scheme 
\begin{equation}
\begin{array}{ll}
X_{k+1}=g(X_k), &X_0=A,\\
E_{k+1}=L_g(X_k,E_k), &E_0=E,
\end{array} \label{eq:coupledCS}
\end{equation}
for computing $f(A)$ and $L_f(A,E)$. Then, given a second direction matrix $D$, the recursive definition \eqref{eq:higherfrechet} of the Fréchet derivative enables the computation of the second Fréchet derivative $L_f^{[2]}(A;E,D)$ by 
\begin{equation}
	L_f^{[2]}(A;E,D)=\mathrm{Im}\left(\frac{L_f(A+ihD,E)}{h}\right)+\mathcal{O}(h^2), \label{eq:secondCS}
\end{equation}  
assuming that $f$ is analytical. Additionaly, the order of the directions $E$ and $D$ is not important and we have $L_f^{[2]}(A;E,D)=L_f^{[2]}(A;D,E)$\cite{highamfrechet2}. However, establishing \eqref{eq:secondCS} for the sign and polar mapping can be done using \eqref{eq:coupledCS} if done correctly. To be more precise, one can show that the coupled CS-iteration 
\begin{equation}
\begin{array}{ll}
\widehat{X}_{k+1}=g(\widehat{X}_k), &\widehat{X}_0=A+ihD,\\
\widehat{E}_{k+1}=L_g(\widehat{X}_k,\widehat{E}_k), &\widehat{E}_0=E,
\end{array} \label{eq:coupledCSnew}
\end{equation}
converges as follows:
\begin{align*}
&\mathrm{Re}\left(\widehat{X}_k\right)\rightarrow f(A) + \mathcal{O}(h^2)\tag{i}\\
&\mathrm{Re}\left(\widehat{E}_k\right)\rightarrow L_f(A,E)+\mathcal{O}(h^2)\tag{ii}\\
&\mathrm{Im}\left(\frac{\widehat{X}_k}{h}\right)\rightarrow L_{f}(A,D)+\mathcal{O}(h^2)\tag{iii}\\
&\mathrm{Im}\left(\frac{\widehat{E}_k}{h}\right)\rightarrow L_{f}^{[2]}(A;D,E)+\mathcal{O}(h^2)\tag{iv}
\end{align*}
In theory, this procedure can be used for third or higher order derivatives as well: Given that one has access to an iterative scheme for computing the first $k$ derivatives of a function, the CS can be used for the iterate converging to $f(A)$ to obtain an iterative scheme for the $(k+1)$-st derivative and any of the first $k$ derivatives exchanging any direction $E_1,\dots,E_k$ with the new direction matrix $E_{k+1}$. If one chooses $g$ to be of the Padé type, automorphism structure can be preserved for $X_k$ in any case.\\
In some cases, it might be beneficial to consider this approach as opposed to the state of the art block-method established by Higham \cite[Alg. 3.6]{highamfrechet2} if one has access to sufficiently high derivatives of $g$ and can evaluate them efficiently.
\section{Numerical experiments}
\begin{figure*}[!h]
	\centering
	\includegraphics[clip, trim=1cm 0.8cm 3.1cm 1.2cm, width=0.49\linewidth]{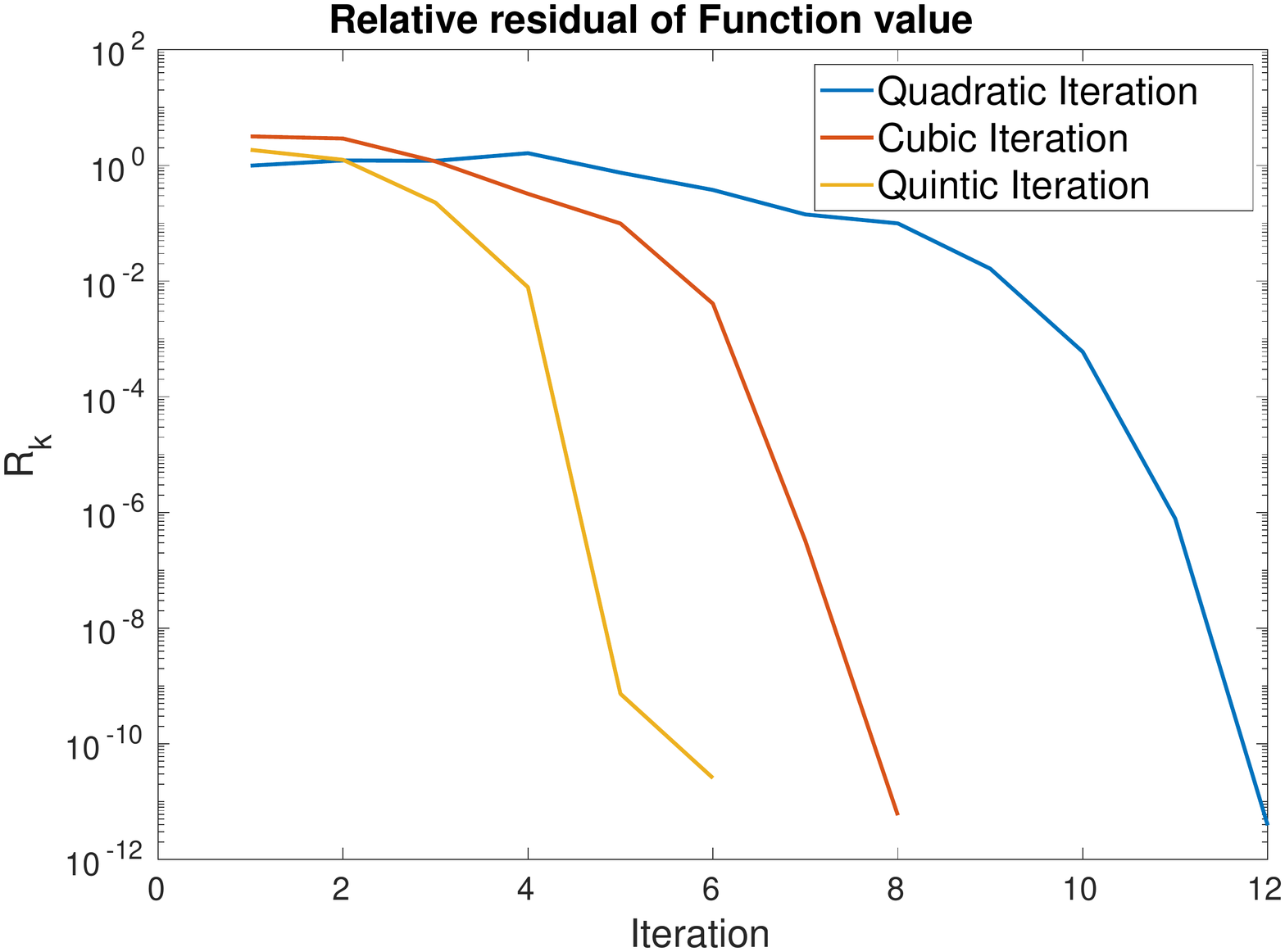}
	\includegraphics[clip, trim=1cm 0.8cm 3.1cm 1.2cm, width=0.49\linewidth]{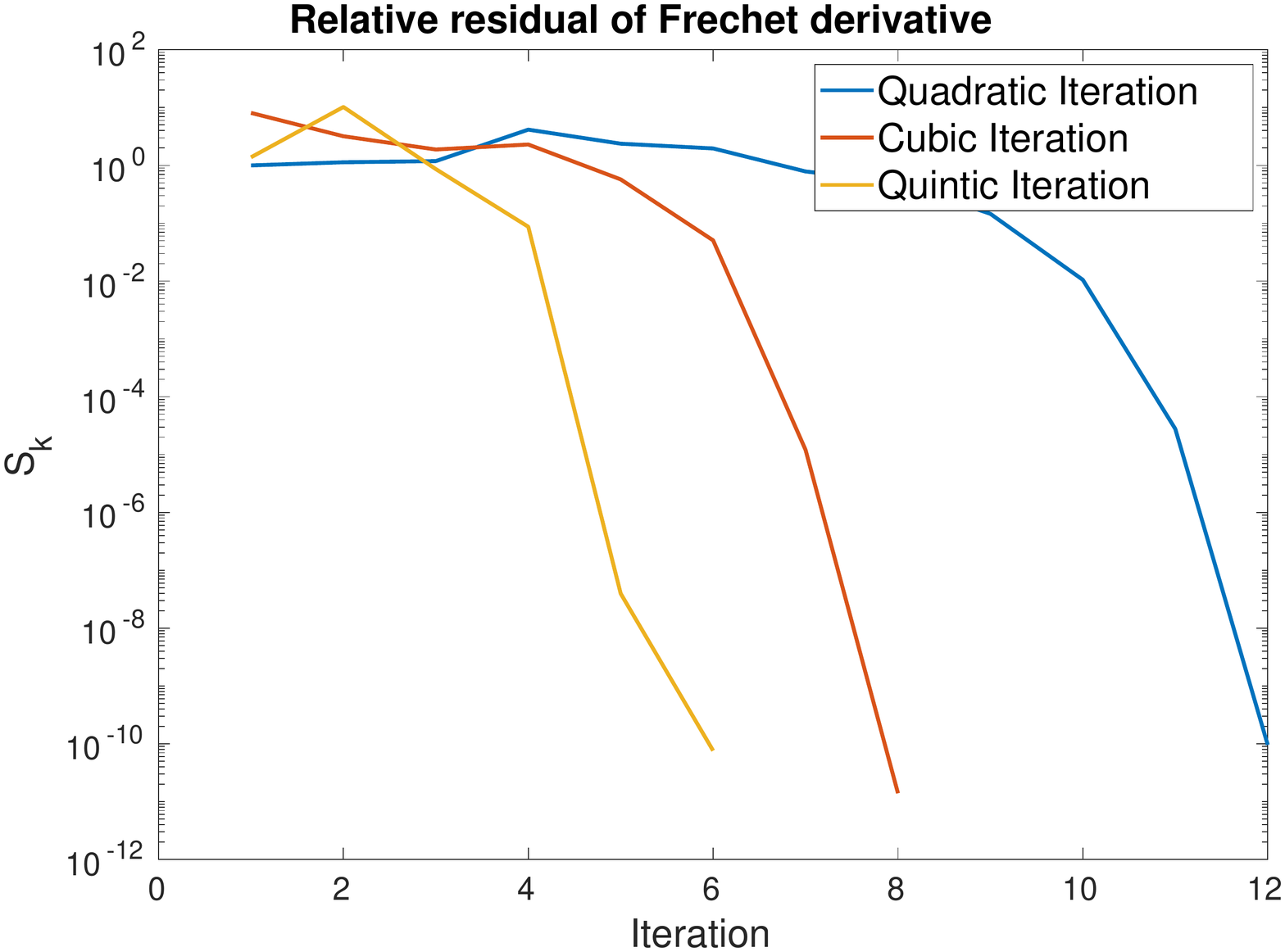}
	\caption{Convergence of function- and Fréchet-iterates for matrix sign}
	\label{fig:iteratesfirstsign}
\end{figure*}
\begin{figure*}[!h]
	\centering
	\includegraphics[clip, trim=1cm 0.8cm 3.1cm 1.2cm, width=0.49\linewidth]{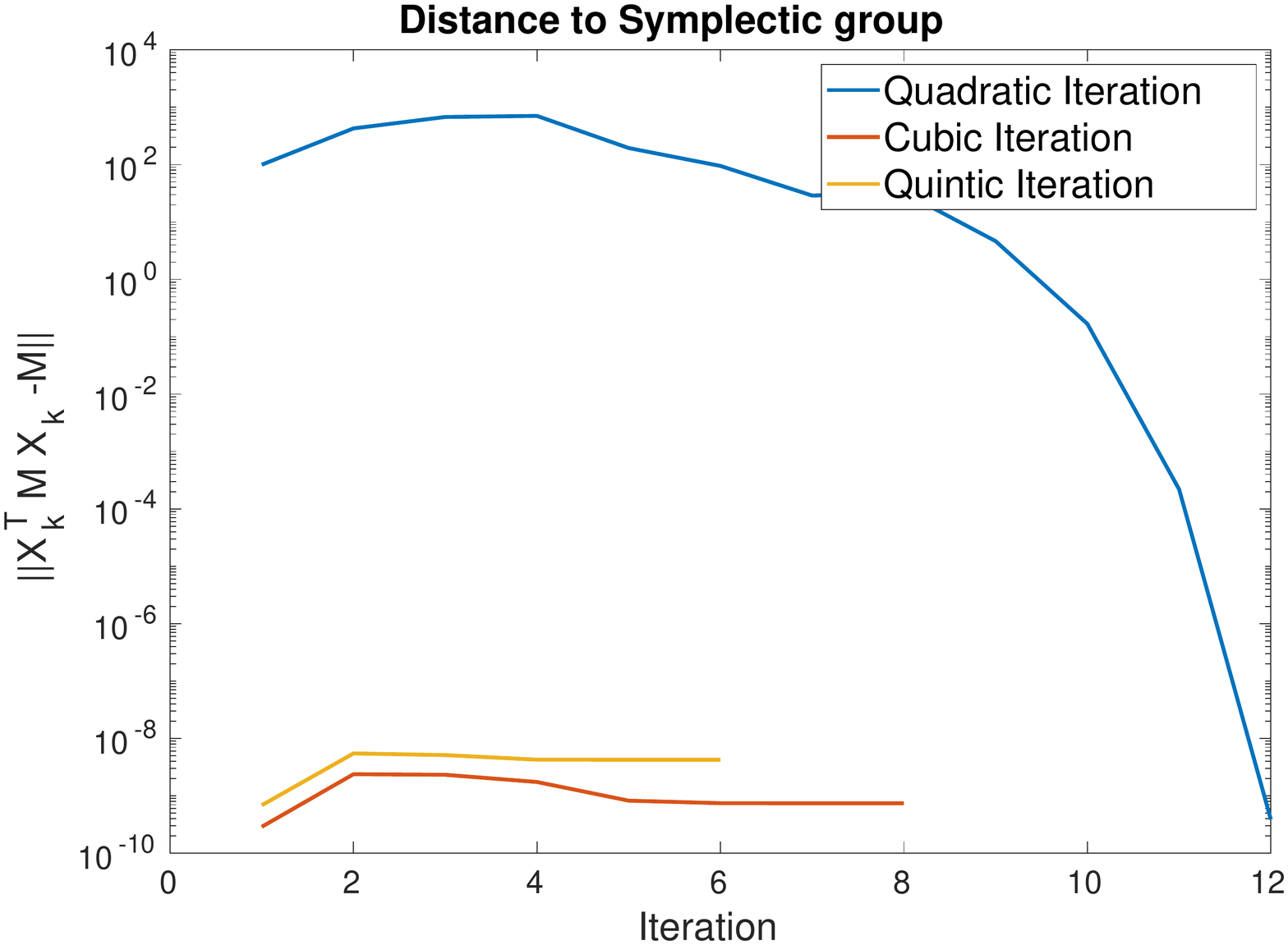}
	\includegraphics[clip, trim=1cm 0.8cm 3.1cm 1.2cm, width=0.49\linewidth]{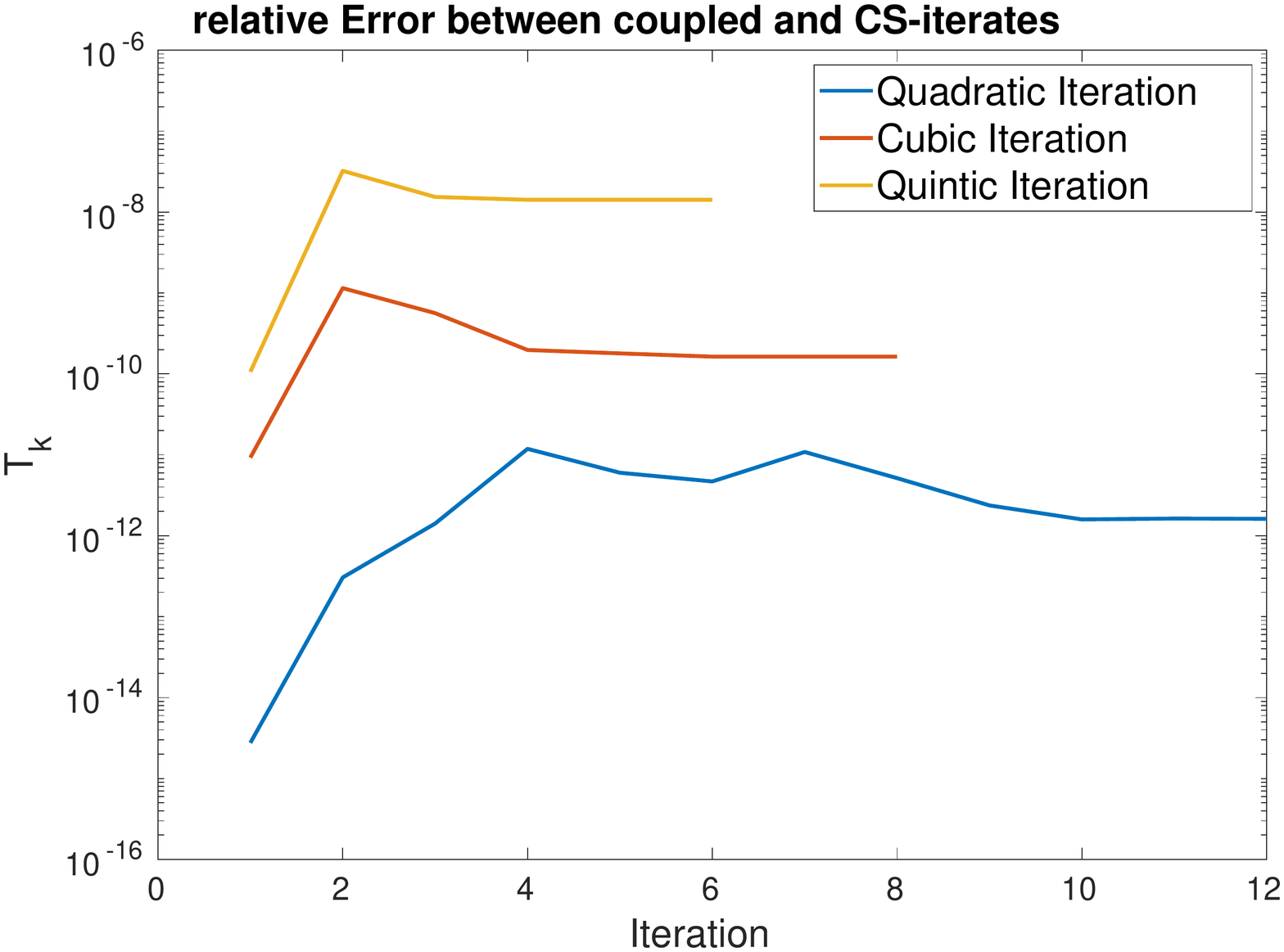}
	\caption{Distance to Symplectic group and difference between coupled and CS-iterates for matrix sign}
	\label{fig:iteratessecondsign}
\end{figure*}
For our experiments, we have tested the quadratic Newton iterations introduced in \Cref{sec:quadnewton} and the cubic and quintic Padé iteration. For the matrix sign, the polynomials $p_{\ell}(I-X^2)$ and $q_{\ell}(I-X^2)$ read
\begin{equation}
	p_{\ell}(I-X^2)=\begin{cases}3I + X^2 & \ell=1\text{ (cubic)}\\
						  5I + 10X^2 + X^4 & \ell=2\text{ (quintic)}\end{cases}	
\end{equation} 
and 
\begin{equation}
	q_{\ell}(I-X^2)=\begin{cases}I + 3X^2 & \ell=1\text{ (cubic)}\\
						  I + 10X^2 + 5X^4 & \ell=2\text{ (quintic)}\end{cases}.	
\end{equation}
We have tested higher order methods ($l=3,4$) as well, but we observed numerical instabilities for $l>2$ since the higher powers of $X$ lead to ill-conditioned systems of equations when trying to solve \eqref{eq:computepade} for $r_{\ell\ell}(I-X^2)$ if the condition number of $X$ exceeds the range of $10^2$ to $10^3$.\\
We use the direct approach from \Cref{sec:direct} for computing $F(A)$ and the derivative $L_F(A,E)$ up to high accuracy and stop our iteration, whenever the relative errors 
\begin{equation}
R_k\coloneqq\frac{||X_k - F(A)||_F}{||F(A)||_F} \text{ and }S_k\coloneqq \frac{||E_k - L_F(A,E)||_F}{||L_F(A,E)||_F} \label{eq:relerrors}
\end{equation}
in $X_k$ and $E_k$ ($\mathrm{Re}(\widehat{X}_k)$ and $\mathrm{Im}\left(\frac{\widehat{X}_k}{h}\right)$, respectively)  satisfy
\begin{equation*}
	R_k < \tau \quad\text{and}\quad S_k < \tau,\quad \tau=10^{-8}.
\end{equation*}
To create a random automorphism of desired size and condition number, we use software designed by Higham and Jagger. Higham's algorithm for pseudo-orthogonal matrices is available through the \texttt{gallery}-command in MATLAB 2021a, the algorithms for symplectic and perplectic matrices that we use can be obtained from \cite{toolbox}. The direction matrix $E$ is chosen randomly using MATLAB's \texttt{rand()}-command.\\
The results displayed in \Cref{fig:iteratesfirstsign} show the convergence behaviour of the quadratic CS-sign iteration \eqref{eq:CSsign} and the corresponding CS-Padé iterations \eqref{eq:CSpadesignthm} for $\ell=m=1$ (cubic) and $\ell=m=2$ (quintic). The matrix $A$ is a random symplectic $400$-by-$400$ matrix with relative condition number \texttt{cond(A)=80}, the CS-stepsize $h$ is chosen to be of order of machine precision ($\approx 10^{-16}$). One can see that the relative error decreases with the same order in the function iterates and the Fréchet iterates in every step and that the Padé iterations converge significantly faster than the standard quadratic iterations, which is due to them having a higher order of convergence. The tolerance of $\tau=10^{-8}$ is achieved after six steps by the quintic iteration, while the quadratic iteration requires twelve steps for convergence. The cubic iteration converges after eight steps.\\
In \Cref{fig:iteratessecondsign}, the absolute error in the automorphism condition $||X_k^T M X_k - M||_F$ and the relative error between the CS-iterates and the coupled iterates are displayed. The benefit of the Padé iteration becomes clear in the left plot, where one can observe that the quadratic iteration converges into the symplectic group when the iterates converge, while the iterates generated by the two Padé-schemes are symplectic up to an absolute error of $10^{-8}$ in every step. On the right hand side, we plot the relative error 
\begin{equation*}
T_k\coloneqq\frac{\left|\left|X_k-\mathrm{Re}\left(\widehat{X}_k\right)\right|\right|_F+\left|\left|E_k-\mathrm{Im}\left(\frac{\widehat{X}_k}{h}\right)\right|\right|_F}{||F(A)||_F + ||L_F(A,E)||_F}
\end{equation*}     
between the CS-iterates and the coupled iterates. Note that \autoref{thm:CSpade} states an error of order $h^2$, which should be around $10^{-32}$ for $h$ being approximately machine precision. For numerical reasons, the error is not as small as expected, nonetheless, the two approaches appear to generate the same iterates up to our desired convergence accuracy of $\tau=10^{-8}$ in the case of the quintic iteration and up to even higher accuracy in the cubic and quadratic case.\\
For all the iterative schemes discussed in this paper, we tested the scaling techniques presented by Higham \cite{matrixfunctions} but did not investigate any significant convergence speed or stability improvements justifying the computational effort required for computing the scaling parameters.
\section{Conclusion and further research}
In this paper, we discussed the sign, square root and polar function and compared a variety of direct and iterative methods for evaluating them. Our main focus was on combining the iterative schemes with the CS to obtain iterative procedures for evaluating the functions and computing the Fréchet derivative simultaneously. We extended the results of Al-Mohy and Higham concerning Newton's method for the matrix sign to the square root and polar function. Using a similar approach, we were then able to combine the CS and the family of Padé iterations, which provided us an iterative method for evaluating the functions and Fréchet derivatives while approximately maintaining automorphism group structure in the process. By proving our main result \Cref{thm:CSpade}, we also used and showed that the CS-iterations are in fact equal to the coupled iterates discussed by Gawlik and Leok (up to $\mathcal{O}(h^2)$). To conclude with our contribution, we put a quick note on how the connection between the coupled and CS-iterations can be used to obtain inexact methods for higher Fréchet derivatives. We finished off with some numerical experiments emphasizing our theoretical results.\\ 
The usage of the Complex Step in combination with iterative schemes can be a point of further research, where some effort can be put into extending our ideas to different matrix functions or iterative approaches and have a closer look on higher derivatives. We will further pursue the application of the CS to matrix valued problems as it provides an easy to implement way of solving Fréchet derivative related problems.

\appendix


\begin{thebibliography}{10}
\bibitem{analytic1}
Mohamed~A. Abul-Dahab and Zanhom~M. Kishka.
\newblock{Analytic functions of complex matrices}.
\newblock{\em Transnational Journal of Mathematical Analysis and Applications}, 2(2):105--130, 2014.

\bibitem{complexstep}
Awad Al-Mohy and Nicholas Higham.
\newblock {The complex step approximation to the Fréchet derivative of a
  matrix function}.
\newblock {\em Numerical Algorithms}, 53, 01 2010.

\bibitem{halley2}
Sergio Amat, Sonia Busquier, and José~M. Gutiérrez.
\newblock Geometric constructions of iterative functions to solve nonlinear
  equations.
\newblock {\em Journal of Computational and Applied Mathematics},
  157(1):197--205, 2003.

\bibitem{applic1}
John Ashburner and Gerard Ridgway.
\newblock Symmetric diffeomorphic modeling of longitudinal structural mri.
\newblock {\em Frontiers in Neuroscience}, 6:197, 2013.

\bibitem{numericalanalysis}
Kendall~E. Atkinson.
\newblock {\em {An introduction to numerical analysis}}.
\newblock John Wiley and Sons, USA, 2 edition, 1991.

\bibitem{finitedifference2}
Russel~R. Barton.
\newblock Computing forward difference derivatives in engineering optimization.
\newblock {\em Engineering Optimization}, 20(3):205--224, 1992.

\bibitem{mehrmann}
Angelika Bunse-Gerstner, Ralph Byers, and Volker Mehrmann.
\newblock {A Chart of Numerical Methods for Structured Eigenvalue Problems}.
\newblock {\em SIAM Journal on Matrix Analysis and Applications},
  13(2):419--453, 1992.

\bibitem{CSNPL}
Maurice~G. Cox and Peter~M. Harris.
\newblock {Numerical analysis for algorithm design in metrology. Software
  Support for Metrology Best Practice Guide No. 11}.
\newblock 2004.

\bibitem{sqrtsecondfrechet}
Pierre Del~Moral and Angele Niclas.
\newblock {A Taylor expansion of the square root matrix function}.
\newblock {\em Journal of Mathematical Analysis and Applications}, 05 2017.

\bibitem{applic3}
Ron Dembo, Stanley Eisenstat, and Trond Steihaug.
\newblock {Inexact Newton Methods}.
\newblock {\em SIAM J. Numer. Anal.}, 19:400--408, 04 1982.

\bibitem{dbi}
Eugene~D. Denman and Alex~N. Beavers.
\newblock {The matrix sign function and computations in systems}.
\newblock {\em Applied Mathematics and Computation}, 2(1):63--94, 1976.

\bibitem{newton}
Luca Dieci.
\newblock {Some numerical considerations and Newton's method revisited for
  solving algebraic Riccati equations}.
\newblock {\em IEEE Transactions on Automatic Control}, 36(5):608--616, 1991.

\bibitem{analytic2}
Ed Doolittle.
\newblock{Analytic Functions of Matrices}.

\bibitem{rectpolar}
Kui Du.
\newblock {The iterative methods for computing the polar decomposition of
  rank-deficient matrix}.
\newblock {\em Applied Mathematics and Computation}, 162(1):95--102, 2005.

\bibitem{complexmatrix}
Adly~T. Fam.
\newblock {Efficient complex matrix multiplication}.
\newblock {\em IEEE Transactions on Computers}, 37(7):877--879, 1988.

\bibitem{polarfrechet}
Evan~S. Gawlik and Melvin Leok.
\newblock {Computing the Fréchet Derivative of the Polar Decomposition}, 2016.

\bibitem{finitedifference1}
Philip~E. Gill, Walter Murray, Michael~A. Saunders, and Margaret~H. Wright.
\newblock Computing forward-difference intervals for numerical optimization.
\newblock {\em SIAM Journal on Scientific and Statistical Computing},
  4(2):310--321, 1983.

\bibitem{kantorovich}
William~B. Gragg and Richard~A. Tapia.
\newblock {Optimal Error Bounds for the Newton-Kantorovich Theorem}.
\newblock {\em SIAM Journal on Numerical Analysis}, 11(1):10--13, 1974.

\bibitem{structurepreserving}
Nicholas Higham, D.~Mackey, Niloufer Mackey, and Françoise Tisseur.
\newblock {Functions Preserving Matrix Groups and Iterations for the Matrix
  Square Root}.
\newblock {\em SIAM Journal on Matrix Analysis and Applications}, 26, 01 2005.

\bibitem{highamfrechet2}
Nicholas Higham and Samuel Relton.
\newblock {Higher Order Fréchet Derivatives of Matrix Functions and the
  Level-2 Condition Number}.
\newblock {\em SIAM Journal on Matrix Analysis and Applications},
  35:1019--1037, 07 2014.

\bibitem{highamsign}
Nicholas~J. Higham.
\newblock {The matrix sign decomposition and its relation to the polar
  decomposition}.
\newblock {\em Linear Algebra and its Applications}, 212-213:3--20, 1994.

\bibitem{matrixfunctions}
Nicholas~J. Higham.
\newblock {\em {Functions of Matrices}}.
\newblock Society for Industrial and Applied Mathematics, 2008.

\bibitem{signpolar}
Nicholas~J. Higham, D.~Steven Mackey, Niloufer Mackey, and Françoise Tisseur.
\newblock {Computing the Polar Decomposition and the Matrix Sign Decomposition
  in Matrix Groups}.
\newblock {\em SIAM Journal on Matrix Analysis and Applications},
  25(4):1178--1192, 2004.

\bibitem{toolbox}
David~P. Jagger.
\newblock {MATLAB Toolbox for Classical Matrix Groups}.
\newblock Master's thesis, University of Manchester, 2003.

\bibitem{applic2}
Ben Jeuris, Raf Vandebril, and Bart Vandereycken.
\newblock {A Survey And Comparison Of Contemporary Algorithms For Computing The
  Matrix Geometric Mean}.
\newblock {\em Electronic transactions on numerical analysis ETNA},
  39:379--402, 01 2012.

\bibitem{rationalsign}
Charles Kenney and Alan~J. Laub.
\newblock {Rational Iterative Methods for the Matrix Sign Function}.
\newblock {\em SIAM Journal on Matrix Analysis and Applications},
  12(2):273--291, 1991.

\bibitem{halley}
Cetin~K. Koc and Bertan Bakkaloglu.
\newblock {Halley's method for the matrix sector function}.
\newblock {\em IEEE Transactions on Automatic Control}, 40(5):944--949, 1995.

\bibitem{CSmoler}
James~N. Lyness and Cleve~B. Moler.
\newblock {Numerical Differentiation of Analytic Functions}.
\newblock {\em SIAM Journal on Numerical Analysis}, 4(2):202--210, 1967.

\bibitem{complexinverse}
Winthrop Smith and Stephen Erdman.
\newblock {A note on the inversion of complex matrices}.
\newblock {\em IEEE Transactions on Automatic Control}, 19(1):64--64, 1974.
\end{thebibliography}
\end{document}